\documentclass[12pt]{article}
\usepackage{graphicx}
\usepackage{amsmath}
\usepackage{amssymb}
\usepackage{theorem}
\usepackage{pgfplots}
\usepackage{enumitem}

\let\originalleft\left
\let\originalright\right
\renewcommand{\left}{\mathopen{}\mathclose\bgroup\originalleft}
\renewcommand{\right}{\aftergroup\egroup\originalright}

\numberwithin{equation}{section}

\newtheorem{thm}{Theorem}[section]
\newtheorem{lemma}[thm]{Lemma}
\newtheorem{prop}[thm]{Proposition}
\newtheorem{cor}[thm]{Corollary}
{\theorembodyfont{\rmfamily}

\newtheorem{rmk}[thm]{Remark}
}

\newcommand{\qed}{\hfill \mbox{\raggedright \rule{.07in}{.1in}}}

\newenvironment{proof}{\vspace{1ex}\noindent{\bf Proof}\hspace{0.5em}}{\hfill\qed\vspace{1ex}}

\newenvironment{pfof}[1]{\vspace{1ex}\noindent{\bf Proof of #1}\hspace{0.5em}}{\hfill\qed\vspace{1ex}}

\newcommand{\R}{{\mathbb R}}

\newcommand{\Z}{{\mathbb Z}}

\newcommand{\N}{{\mathbb N}}

\newcommand{\bP}{\mathbb{P}}
\newcommand{\bE}{\mathbb{E}}

\newcommand{\cA}{\mathcal{A}}
\newcommand{\cB}{\mathcal{B}}
\newcommand{\cM}{\mathcal{M}}

\newcommand{\eps}{{\epsilon}}
\newcommand{\diam}{\operatorname{diam}}
\newcommand{\SMALL}{\textstyle}
\newcommand{\BIG}{\displaystyle}

\title{Explicit Coupling Argument for Nonuniformly Hyperbolic Transformations}

\author{A. Korepanov \hspace{2em}  
Z. Kosloff\footnote{Permanent address: 
Einstein Institute of Mathematics,
Hebrew University of Jerusalem, 
Givat Ram. Jerusalem, 9190401, Israel}
\hspace{2em} I. Melbourne \\[.75ex]
 {\small Mathematics Institute, University of Warwick, Coventry, CV4 7AL, UK}}

\date{26 September 2016
(updated 12 June 2018)}

\begin{document}

 \maketitle

\begin{abstract}
The transfer operator corresponding to a uniformly expanding map enjoys good spectral properties.  Here it is verified that coupling yields explicit estimates that depend continuously on the expansion and distortion constants of the map.

For nonuniformly expanding maps with a uniformly expanding induced map, we obtain explicit estimates for mixing rates (exponential, stretched exponential, polynomial) that again depend continuously on the constants for the induced map together with data associated to the inducing time.

Finally, for nonuniformly hyperbolic transformations, we obtain the corresponding estimates for rates of decay of correlations.
\end{abstract}

\section{Introduction}

It is well-known that the transfer operator associated to a uniformly expanding map enjoys good spectral properties.  
In particular, there are numerous methods for proving exponential decay of correlations for uniformly expanding maps, see for example~\cite{Aaronson,Bowen75,ParryPollicott90,Ruelle78,Sinai72}.

Often, statistical properties of nonuniformly expanding systems are studied by inducing to a uniformly expanding one.  
Young~\cite{Young98,Young99} obtained results on decay of correlations for large classes of such nonuniformly expanding maps, as well as nonuniformly hyperbolic transformations.   The rate of decay is related to the tails of the inducing time, with special emphasis placed on exponential tails and polynomial tails.
Stretched exponential decay rates (amongst others) were obtained in
Maume-Deschamps~\cite{Maume01}.
The resulting decay rates have the form $O(e^{-cn^\gamma})$ or $O(n^{-\beta})$
where $\gamma\in(0,1]$ and $\beta>0$ are given explicitly, but the 
implied constants are not and nor is $c$ in the exponential case $\gamma=1$.
An improved estimate of Gou\"ezel~\cite{GouezelPhD} gives sharp decay rates in the stretched exponential case $\gamma\in(0,1)$ but the implied constant remains nonexplicit (as does the constant $c$ in the exponential case).

In this paper, we use an explicit coupling argument to obtain mixing rates with uniform control on the various constants.  The main novelty in our results lies in the nonuniformly expanding/hyperbolic setting.  However, even for uniformly expanding maps, we expect that our results have numerous applications, see for example~\cite{KKMsub1,KKMsub2}.

Related results using the coupling method for uniformly expanding maps can be found in both simpler and more complicated situations (usually in low dimensions) in recent papers, for example~\cite{Eslamiapp,Sulku}.  See also~\cite{Liverani01} for an approach using Birkhoff cones for one-dimensional maps.  None of these results are formulated in such a way that they can be cited in~\cite{KKMsub1,KKMsub2}.  In this paper, we work in a general metric space and present a much shorter and more elementary proof than was previously written down.  
The results then feed into the more complicated argument required in
the nonuniformly expanding/hyperbolic setting.

\begin{rmk}
After circulating a first version of this paper, we were made aware
by Oliver Butterley and Jean-Ren\'e Chazottes of
previous work of Zweim\"uller~\cite{Zweimuller04} which handles the uniformly expanding case.
Using a coupling argument for uniform expanding Markov maps
defined on a general compact metric space,~\cite{Zweimuller04} shows how to obtain
exponential decay of correlations with explicit control on the various
constants, just as is shown in this paper.
Moreover, the setting in~\cite{Zweimuller04} (within the uniformly expanding setting) is
more general than the one considered here since we assume full branches whereas~\cite{Zweimuller04} assumes a ``finite images'' condition.
Assuming full branches simplifies matters considerably but suffices for
our purposes in~\cite{KKMsub1,KKMsub2}.

The compactness assumption in~\cite{Zweimuller04} is used only to
to prove existence of an invariant density via an Arzel\`a-Ascoli argument.
The proof below of Proposition~\ref{prop-inv} shows how to bypass this, so that
compactness of the metric space is not required.
For an alternative argument to prove existence of an invariant density without using compactness, see~\cite{ADU93} or~\cite[Lemma 4.4.1]{Aaronson}.

Hence our results for uniformly expanding maps in Section~\ref{subsec-UE} are not new.
We include the results for a number of reasons: (a) completeness, especially as they feed into our results for nonuniformly expanding/hyperbolic systems (Sections~\ref{subsec-NUE} and~\ref{subsec-NUH}) which are new; (b) The arguments are very short and direct; (c) The explicit nature of the constants is stated in a way that is convenient for easy reference (in [32] it is necessary to read the entire proof to see that it gives explicit uniform bounds for the constants).
\end{rmk}


\begin{rmk} Keller \& Liverani~\cite{KellerLiverani99} considered continuous families of  uniformly expanding maps and developed a perturbative theory that gives uniform estimates on the spectra of the associated transfer operators.  
This idea was used by~\cite{DemersZhang13} in the situation of dispersing billiards.
However, inducing from continuous families of nonuniformly expanding maps to families of uniformly expanding maps may fail to preserve any useful notion of continuous dependence.
In particular, the examples in~\cite[Section~5]{KKMsub1} and in~\cite{KKMsub2} do not satisfy the hypotheses of~\cite{DemersZhang13,KellerLiverani99}.

In this paper, we do not assume any continuous dependence on parameters.
Instead, we work with a fixed uniformly expanding map $F$, and give explicit estimates on the associated transfer operator that depend continuously 
on the expansion and distortion estimates of $F$.
\end{rmk}

Even for nonuniformly expanding/hyperbolic dynamical systems, none of the results in this paper are particularly surprising.  Nevertheless, the results go far beyond those previously available.
Some examples are listed at the end of Section~\ref{subsec-NUE}.
In the case of smooth unimodal maps there are previous results~\cite[Theorem~1.3]{BaladiBenedicksSchnellmann15} showing exponential decay of correlations up to a finite period with uniform exponent (uniformity of the implied constant is not claimed in~\cite{BaladiBenedicksSchnellmann15}). Here we obtain a similar result with uniform exponent and uniform implied constant.
In the case of families of Viana maps~\cite{Viana97} which are known to have stretched exponential decay of correlations~\cite{Gouezel06}, we obtain for the first time 
uniform estimates on the constants $C,\,c,\,\gamma$ in the
stretched exponential decay rate $Ce^{-cn^\gamma}$.

Our main results are stated in Section~\ref{sec-statement} and proved 
for uniformly expanding, nonuniformly expanding, and nonuniformly hyperbolic, transformations in
Sections~\ref{sec-UE},~\ref{sec-NUE} and~\ref{sec-NUH} respectively.

\section{Statement of the main results}
\label{sec-statement}

In this section, we state our main results for
uniformly expanding maps (Subsection~\ref{subsec-UE}),
nonuniformly expanding maps (Subsection~\ref{subsec-NUE}),
and nonuniformly hyperbolic transformations (Subsection~\ref{subsec-NUH}).

\subsection{Uniformly expanding maps}
\label{subsec-UE}

Let $(Y,m)$ be a probability space, and 
$F:Y \to Y$ be a nonsingular transformation.
Let $d$ be a metric on $Y$ such that $\diam Y\leq 1$.

Suppose that $\alpha$ is an at most countable measurable partition of $Y$, and
that $F$ restricts to a measure-theoretic bijection from $a$ onto $Y$ for each $a \in \alpha$.

Let $\zeta=\frac{dm}{dm\circ F}$ be the inverse Jacobian of $F$ with respect 
to $m$.
Assume that there are constants $\lambda>1$, $K>0$ and $\eta\in(0,1]$ such that
for $x,y$ in the same partition element
\begin{align}
  d(Fx,Fy) \geq \lambda d(x,y)
  \quad \text{and} \qquad
  |\log \zeta(x) - \log \zeta(y)|  \leq K  d(Fx,Fy)^\eta. \label{eq-zeta-dist}
\end{align}

Let $P_m:L^1(Y)\to L^1(Y)$ be the transfer operator corresponding
to $F$ and $m$, so 
$\int_Y P_m\phi\,\psi\,dm=\int_Y\phi\,\psi\circ F\,dm$
for all $\phi\in L^1$ and $\psi\in L^\infty$.
Then $P_m \phi$ is given explicitly by
\[
  (P_m \phi)(y) = \sum_{a \in \alpha} \zeta(y_a) \phi(y_a),
\]
where $y_a$ is the unique preimage of $y$ under $F$ lying in $a$.

Given $\phi:Y\to\R$, define 
\[
  |\phi|_\eta = \sup_{x\neq y} \frac{|\phi(x)-\phi(y)|}{d(x,y)^\eta}
  \qquad \text{and} \qquad \|\phi\|_\eta = |\phi|_\infty + |\phi|_\eta.
\]
Let $C^\eta$ denote the Banach space of observables $\phi:Y\to\R$
such that $\|\phi\|_\eta<\infty$.

It is well-known that there exist constants
$C>0$, $\gamma\in(0,1)$, such that
$\|P_m^n\phi\|_\eta \le C\gamma^n\|\phi\|_\eta$
for all $\phi\in C^\eta$ with $\int_Y\phi\,dm=0$ and all $n\ge1$.
Our main result is:

\begin{thm} \label{thm-UE}
There exist constants $C>0$, $\gamma\in(0,1)$ depending continuously
on $\lambda$, $K$ and $\eta$, such that
\[
\|P_m^n\phi\|_\eta \le C\gamma^n|\phi|_\eta,
\]
for all $\phi\in C^\eta$ with $\int_Y\phi\,dm=0$, and all $n\ge1$.
\end{thm}

\begin{rmk} \label{rmk-main}
For example, take
$R = 2K/(1-\lambda^{-\eta})$ and $\xi = \frac12 e^{-R}(1-\lambda^{-\eta})$.
Then Theorem~\ref{thm-UE} holds with $C=4 e^R (1+R)$ and $\gamma=1-\xi$.
\end{rmk}

Next, let $\cM$ be the collection of probability measures on $Y$ that
are equivalent to $m$ and satisfy
$L_\mu<\infty$ where
$L_\mu=|\log\frac{d\mu}{dm}|_\eta$.
Given $\mu\in\cM$, 
define $\zeta_\mu=\frac{d\mu}{d\mu\circ F}$ and let $P_\mu$ be the corresponding transfer operator.

\begin{prop} \label{prop-mu}
For all $x,y$ in the same partition element,
\[
|\log \zeta_\mu(x)-\log\zeta_\mu(y)|\le K_\mu d(Fx,Fy)^\eta,
\]
where $K_\mu=K+(\lambda^{-\eta}+1)L_\mu$.  
\end{prop}

\begin{proof}
Note that $\log \zeta_\mu=\log \zeta+h-h\circ F$ where $h=\log\frac{d\mu}{dm}$.
Hence $|\log \zeta_\mu(x)-\log \zeta_\mu(y)|\le |\log \zeta(x)-\log \zeta(y)|+|h|_\eta d(x,y)^\eta+
|h|_\eta d(Fx,Fy)^\eta\le (K+L_\mu \lambda^{-\eta}+L_\mu)d(Fx,Fy)^\eta$.
\end{proof}

In other words, the hypotheses of Theorem~\ref{thm-UE} are satisfied with $m$ and $K$ replaced by $\mu$ and $K_\mu$.
Hence, we obtain:

\begin{cor} \label{cor-main}
Let $\mu\in\cM$.
There exist constants $C>0$, $\gamma\in(0,1)$ depending continuously
on $\lambda$, $K_\mu$ and $\eta$, such that
\[
\|P_\mu^n\phi\|_\eta \le C\gamma^n|\phi|_\eta,
\]
for all $\phi\in C^\eta$ with $\int_Y\phi\,d\mu=0$, and all $n\ge1$. \qed
\end{cor}

Of special interest is the case where $\mu$ is the unique 
absolutely continuous $F$-invariant probability measure.
For this special case, we prove:

\begin{prop} \label{prop-inv}  
The invariant probability measure $\mu$ lies in $\cM$, and 
there is a constant $R$ depending continuously on $\lambda$, $K$ and $\eta$
(chosen as in Remark~\ref{rmk-main} say) such that
\[
e^{-R}\le \frac{d\mu}{dm}\le e^R, \qquad
\Bigl|\log\frac{d\mu}{dm}\Bigr|_\eta\le R.
\]
In particular, the constants $C$ and $\gamma$ in Corollary~\ref{cor-main} depend continuously on $\lambda$, $K$ and $\eta$.
\end{prop}

\begin{rmk}
A standard extension of these results is to treat observables $\phi:Y\to\R$ 
that 
are piecewise H\"older (relative to the partition $\alpha$) and possibly unbounded.  Provided $P_m\phi\in C^\alpha$, our results go through unchanged (with obvious modifications to the constant $C$).
For instances of this extension, we refer to~\cite[Lemma~2.2]{MN05} 
or~\cite[Proposition~4.7]{KKMsub1}.
\end{rmk}

\subsection{Nonuniformly expanding maps}
\label{subsec-NUE}

Let $F:Y\to Y$ be a uniformly expanding map with probability measure $m$ (not necessarily invariant), constants $\lambda$, $K$ and $\eta$, and partition $\alpha$, as in Subsection~\ref{subsec-UE}.
Let \(\tau : Y \to \Z^+\) be an integrable function that is constant on
partition elements. 
Define the Young tower~\cite{Young99}
\[
  \Delta = \{ (y,\ell) \in Y \times \Z : 0 \leq \ell \le  \tau(y)-1 \}
\]
and \(f : \Delta \to \Delta\),
\[
  f(y,\ell) = \begin{cases} (y, \ell+1), & \ell \le  \tau(y)-2, \\ (Fy, 0), & \ell=\tau(y)-1. \end{cases}
\]
Let \(\bar{\tau} = \int_Y \tau \, dm\).
Let \(m_\Delta\) be the probability measure on \(\Delta\) given by
\(
  m_\Delta(A \times \{\ell\}) = \bar{\tau}^{-1} m(A)
\)
for all \(\ell \geq 0\) and measurable $A \subset \{y\in Y:\tau(y)\ge \ell+1\}$.

Let \(d_\Delta\) be the metric on \(\Delta\) given by
\[
  d_\Delta((y,\ell), (y',\ell')) = \begin{cases} 1, &\ell \neq \ell' \\ d(y,y'), & \ell=\ell' \end{cases}.
\]

Given \(\phi : \Delta \to \R\), define
\(|\phi|_\eta = \sup_{x,y \in \Delta} \frac{|\phi(x) - \phi(y)|}{d_\Delta(x,y)^\eta}\)
and \(\|\phi\|_\eta = |\phi|_\eta + |\phi|_\infty \).

Let \(L: L^1(\Delta) \to L^1(\Delta)\) denote the 
transfer operator corresponding to \(f\) and \(m_\Delta\),
so $\int_\Delta L\phi\,\psi\,dm_\Delta=\int_\Delta \phi\,\psi\circ f\,d\mu$
for all $\phi\in L^1$, $\psi\in L^\infty$.

When the measure $m$ on $Y$ is $F$-invariant, $m_\Delta$ is an 
ergodic $f$-invariant probability measure on $\Delta$ and $m_\Delta$ is mixing under $f$ if and only 
if $\gcd\{\tau(a):a\in\alpha\}=1$.
Accordingly, we say that the tower $f:\Delta\to\Delta$ is {\em mixing}
if $\gcd\{\tau(a):a\in\alpha\}=1$, and {\em nonmixing} otherwise, even though
we do not assume that $m_\Delta$ is $f$-invariant.

\paragraph{Mixing Young towers} 
In the mixing case, there exist
\(\delta > 0\) and a finite set of positive integers \(\{I_k\}\)
with \(\gcd\{I_k\} = 1\) 
such that \(m ( \{y \in Y : \tau(y) = I_k\}) \geq \delta \).

\begin{thm}
  \label{thm-NUE}
  Let \(\phi : \Delta \to \R\) be an observable with \(\|\phi\|_\eta < \infty\)
  and \(\int_\Delta \phi \, dm_\Delta = 0\).
  \begin{itemize}
    \item Suppose that \(m(\tau \geq n) \leq C_\tau n^{-\beta} \) for some \(\beta > 1\)
      and all \(n > 0\).
      Then there exists a constant \(C>0\) depending continuously on
      \(\lambda\), \(K\), \(\eta\), \(\max \{I_k\}\), \(\delta\), \(\beta\)
      and \(C_\tau\),  such that for all \(n \geq 0\)
      \[
        \int_\Delta |L^n \phi| \, dm_\Delta \leq C \|\phi\|_\eta n^{-(\beta-1)}.
      \]
    \item Suppose that \(m(\tau \geq n) \leq C_\tau e^{- A n^\gamma}\) for some 
      \(A > 0\), \(0 < \gamma \leq 1\) and all \(n > 0\). 
      Then there exist constants \(B > 0\) and \(C > 0\) depending continuously
      on \(\lambda\), \(K\), \(\eta\), \(\max \{I_k\}\), \(\delta\), \(A\), \(\gamma\)
      and \(C_\tau\), such that for all \(n \geq 0\)
      \[
        \int_\Delta |L^n \phi| \, dm_\Delta \leq C \|\phi\|_\eta e^{-B n^\gamma}.
      \]
  \end{itemize}
\end{thm}

\paragraph{Nonmixing Young towers} 
In the nonmixing case, define
\[
d=\gcd\{j\ge1:m(\{y\in Y:\tau(y)=j\})>0\}\ge2.
\]
There exist
\(\delta > 0\) and a finite set of positive integers \(\{I_k\}\)
with \(\gcd\{I_k\} = d\) 
such that \(m ( \{y \in Y : \tau(y) = I_k\}) \geq \delta \).

\begin{thm}
  \label{thm-NUE2}
  Let \(\phi : \Delta \to \R\) be an observable with \(\|\phi\|_\eta < \infty\)
  and \(\int_\Delta \phi \, dm_\Delta = 0\).
  Then Theorem~\ref{thm-NUE} holds with
  \(\int_\Delta |L^n \phi| \, dm_\Delta\) replaced by
  \(
    \int_\Delta \bigl| \sum_{k=0}^{d-1} L^{nd+k} \phi \bigr| \, dm_\Delta
    .
  \)
\end{thm}

Theorem~\ref{thm-NUE2} has the following equivalent reformulation which gives uniform mixing rates up to a cycle of length $d$.  We state the reformulation for the case of (stretched) exponential mixing.  The polynomial mixing case goes the same way.

Write $\Delta=E_1\cup \dots\cup E_d$ where $f(E_j)=E_{j+1\bmod d}$
and $f^d:E_j\to E_j$ is a mixing tower for $j=1,\dots,d$.

\begin{cor}  \label{cor-NUE2}
Suppose that we are in the situation of Theorem~\ref{thm-NUE2} and that $m(\tau\ge n)\le C_\tau e^{-An^\gamma}$ as in the second part of Theorem~\ref{thm-NUE}.
Fix $j=1,\dots,d$.
Then there exist uniform constants $B,\,C>0$ as in Theorem~\ref{thm-NUE} such that
\[
\Big|\int_\Delta \phi\,\psi\circ f^{nd}\,dm_\Delta
-\int_\Delta \phi\,dm_\Delta
\int_\Delta \psi\,dm_\Delta\Big|\le C\|\phi\|_\eta|\psi|_\infty e^{-Bn^\gamma},
\]
for all $n\ge1$ and all $\phi,\psi\in L^\infty$ supported in $E_j$ with
$\|\phi\|_\eta<\infty$. 
\end{cor}

\paragraph{Examples}

In~\cite{KKMsub1,KKMsub2}, we verified for specific families of nonuniformly expanding maps that the corresponding
induced maps $F$ are uniformly expanding, as in Subsection~\ref{subsec-UE},
with uniform constants $\lambda,K,\eta$.
A key ingredient in this verification is the work of~\cite{Alves04,AlvesLuzzattoPinheiro05,AlvesViana02,FreitasTodd09} on strong statistical stability
(where the density of the invariant measure varies continuously in $L^1$).
It follows  from this abstract framework (specifically condition (U1) in~\cite{AlvesViana02}) that the data $d=\gcd\{I_k\}\ge1$ and $\delta>0$ associated with the inducing time $\tau$ varies continuously in the mixing case and upper semicontinuously in general (so $d$ can decrease under small perturbations but cannot increase).
Hence for the examples in~\cite{KKMsub1,KKMsub2}, uniform estimates
on decay of correlations follow immediately from
Theorems~\ref{thm-NUE} and~\ref{thm-NUE2}.

Specifically, we obtain uniform polynomial decay of correlations for
intermittent maps~\cite[Example~4.9]{KKMsub2},
uniform exponential decay of correlations (up to a finite cycle) for smooth unimodal and multimodal maps satisfying the Collet-Eckmann condition~\cite[Example~4.10]{KKMsub2}, and
uniform stretched exponential decay of correlations for Viana maps~\cite[Example~4.11]{KKMsub2}.

\subsection{Nonuniformly hyperbolic transformations}
\label{subsec-NUH}

Let $T:M\to M$ be a diffeomorphism (possibly with singularities) defined on a
Riemannian manifold $(M,d)$.   
Fix a subset $Y\subset M$.  It is assumed that there is a ``product structure'':
namely a family of ``stable disks'' $\{W^s\}$ that are disjoint and cover $Y$,
and a family of ``unstable disks'' $\{W^u\}$ that are disjoint and cover $Y$.
Each stable disk intersects each unstable disk in precisely one point.
The stable and unstable disks containing $y$ are labelled $W^s(y)$ and $W^u(y)$.

Suppose that there is a partition $\{Y_j\}$ of $Y$ and
integers $\tau(j)\ge1$ with $\gcd\{\tau(j)\}=1$ such that $T^{\tau(j)}(W^s(y))\subset W^s(T^{\tau(j)}y)$
for all $y\in Y_j$.
Define the return time function $\tau:Y\to\Z^+$ by $\tau|_{Y_j}=\tau(j)$
and the induced map
$F:Y \to Y$ by $F(y)=T^{\tau(y)} (y)$.

Let $s$ denote the {\em separation time} with respect to the map $F:Y\to Y$.  
That is, if
$y,z\in Y$, then $s(y,z)$ is the least integer $n\ge0$ such that $F^n x$, $F^ny$ lie in distinct partition elements of $Y$.    
\begin{itemize}
\item[(P1)]
There exist constants $K_0\ge1$, $\rho_0\in(0,1)$ such that
\begin{itemize}
\item[(i)]  If $z\in W^s(y)$, then $d(F^ny,F^nz)\le K_0\rho_0^n$, 
\item[(ii)]  If $z\in W^u(y)$, then $d(F^ny,F^nz)\le K_0\rho_0^{s(y,z)-n}$,
\item[(iii)] If $y,z\in Y$, then $d(T^jy,T^jz)\le K_0(d(y,z)+d(Fy,Fz))$
for all $0\le j<\min\{\tau(y),\tau(z)\}$.
\end{itemize}
\end{itemize}

Let $\bar Y=Y/\sim$ where $y\sim z$ if $y\in W^s(z)$ and
define the partition $\{\bar Y_j\}$ of $\bar Y$.
We obtain a well-defined return time function  $\tau:\bar Y\to\Z^+$ and
induced map $\bar F:\bar Y\to\bar Y$.
Suppose that
the map $\bar F:\bar Y\to\bar Y$ and partition $\alpha=\{\bar Y_j\}$ separate points in $\bar Y$,
and let $s$ denote also the separation time on $\bar{Y}$.
Fix $\theta\in(0,1)$.
Then $d_\theta(y,z)=\theta^{s(y,z)}$ defines
a metric on $\bar Y$.
Suppose further that  
$\bar F:\bar Y\to\bar Y$ is a uniformly expanding map
in the sense of Subsection~\ref{subsec-UE} on the metric space $(\bar Y,d_\theta)$,
with partition $\alpha$ and constants $\lambda=1/\theta>1$, $K>0$, $\eta=1$.
Let $\bar\mu_Y$ denote the $\bar F$-invariant probability measure on $\bar Y$ from Proposition~\ref{prop-inv}.
We assume that $\tau:\bar Y\to\Z^+$ is integrable.
We suppose also that there is an
$F$-invariant probability measure 
$\mu_Y$ on $Y$ such that $\bar\pi_*\mu_Y=\bar{\mu}_{Y}$ where $\bar\pi:Y\to\bar Y$ is the quotient map.

As in Subsection~\ref{subsec-NUE},
starting from $\bar F:\bar Y\to \bar Y$ and $\tau:\bar Y\to\Z^+$, we can form the {\em quotient tower} $\bar f:\bar\Delta\to\bar \Delta$ with
$\bar f$-invariant mixing probability measure $\bar\mu_\Delta$.
Similarly, starting from $F:Y\to Y$ and $\tau:Y\to\Z^+$, we form the tower 
$f:\Delta\to\Delta$ such that $F=f^\tau:Y\to Y$ with
$f$-invariant mixing probability measure $\mu_\Delta$.

Define the semiconjugacy
$\pi:\Delta\to M$, $\pi(y,\ell)=T^\ell y$.  Then
$\mu=\pi_*\mu_\Delta$ is a $T$-invariant mixing probability measure on $M$.

As in Subsection~\ref{subsec-NUE}, we restrict to the cases $\mu(\tau>n)=O(n^{-\beta})$, $\beta>1$, and $\mu(\tau>n)=O(e^{-An^\gamma})$, $A>0$, $\gamma\in(0,1]$.

\begin{thm} \label{thm-NUH}
Let $\eta\in(0,1]$.
 Then there exist $C>0$, $B>0$ depending continuously on the constants in Theorem~\ref{thm-NUE} (associated to the nonuniformly expanding map
$\bar f:\bar\Delta\to\bar\Delta$)  as well as $\eta$, $\rho_0$ and $K_0$, such that 
$|\int_M v\,w\circ T^n\,d\mu-\int_M v\,d\mu\int_M w\,d\mu|
\le  Ca_n\|v\|_\eta\|w\|_\eta$,
 for all $v,w\in C^\eta(M)$, $n\ge 1$,
where $a_n=n^{-(\beta-1)}$ or $e^{-Bn^\gamma}$ respectively.
 \end{thm}

\begin{rmk} \label{rmk-NUH}
Note that there is no assumption about contraction rates along stable manifolds for $T$; all that is required is exponential contraction/expansion for the induced map $F:Y\to Y$.
This is in contrast to~\cite{Young98} where exponential contraction is assumed for $T$ (this restriction is also present 
in~\cite{AlvesPinheiro08}) and~\cite{AlvesAzevedo16} where polynomial contraction is assumed for $T$.

The method for removing such assumptions on contractivity of $T$ is due to Gou\"ezel (based on ideas in~\cite{ChazottesGouezel12})
and was used previously in~\cite[Appendix~B]{MT14}.
\end{rmk}

\section{Proof for uniformly expanding maps}
\label{sec-UE}

In this section, we prove Theorem~\ref{thm-UE} and Proposition~\ref{prop-inv}.

For $\psi:Y \to (0, \infty)$, we define $|\psi|_{\eta,\ell} = |\log \psi|_\eta$.
Note that
  \begin{align} \label{eq-bound}
    e^{-|\psi|_{\eta,\ell}} \int_Y\psi \, dm
    \le \psi \le 
    e^{|\psi|_{\eta,\ell}}\int_Y\psi\,dm.
  \end{align}
  Also, for at most countably many observables $\psi_k:Y\to(0,\infty)$,
  \begin{align} \label{eq-sum}
    \Bigl|\sum_{k} \psi_k\Bigr|_{\eta,\ell}
    \le \sup_{k}|\psi_k|_{\eta,\ell}.
  \end{align}

\begin{prop} 
  \label{prop-dist}
  Let $\psi:Y\to(0,\infty)$.
  Then
$|P_m\psi|_{\eta,\ell} \le K+ \lambda^{-\eta}|\psi|_{\eta,\ell}$.  
\end{prop}

\begin{proof}
  For $a\in\alpha$ write $\psi_a=1_a\psi$.
  Then $P_m \psi = \sum_{a} P_m \psi_a$.
  For $y\in Y$, we have $(P_m\psi_a)(y)=\zeta(y_a)\psi(y_a)$ 
  where $y_a$ is the unique preimage of $y$ under $F$ lying in $a$.

  Let $x,y\in Y$ with preimages $x_a,y_a\in a$.
  Then
  \begin{align*}
    |\log (P_m\psi_a)(x)- & \log (P_m\psi_a)(y)|
    \le 
    |\log \zeta(x_a)-\log \zeta(y_a)|+ |\log \psi(x_a)-\log \psi(y_a)|
    \\ 
    & \le 
    K d(Fx_a,Fy_a)^\eta+ |\psi|_{\eta,\ell}\, d(x_a,y_a)^\eta
    \le (K+\lambda^{-\eta}|\psi|_{\eta,\ell})d(x,y)^\eta,
  \end{align*}
  and so $|P_m\psi_a|_{\eta,\ell} \le K+\lambda^{-\eta}|\psi|_{\eta,\ell}$.
  The result follows from~\eqref{eq-sum}.
\end{proof}

\begin{prop}\label{prop-diet}
  Let $\psi:Y\to(0,\infty)$.
  For each $t \in [0, e^{-|\psi|_{\eta,\ell}})$
  \[
    \Bigl| \psi - t \int_Y\psi\,dm \Bigr|_{\eta,\ell} 
    \leq \frac{|\psi|_{\eta,\ell}}{1-t e^{|\psi|_{\eta,\ell}}}.
  \]
\end{prop}

\begin{proof}
  Let $\kappa(y) = \log \psi(y)$. Note that 
  \[
    \frac{d}{d \kappa} \log \Bigl(e^\kappa - t \int_Y\psi\,dm\Bigr)
    = \frac{e^\kappa}{e^\kappa - t \int_Y\psi\,dm}
    = \frac{1}{1-t e^{-\kappa}\int_Y\psi\,dm }.
  \]
  By~\eqref{eq-bound},
  \[
   \frac{1}{1-t e^{-\kappa(y)}\int_Y\psi\,dm }
   =\frac{1}{1-t \psi(y)^{-1}\int_Y\psi\,dm }
   \leq \frac{1}{1-t e^{|\psi|_{\eta,\ell}}},
  \]
  for all $y\in Y$.
  Hence, by the mean value theorem, for $x,y\in Y$,
  \begin{align*}
    \Bigl|\log \Bigl( e^{\kappa(x)} - t \int_Y\psi\,dm\Bigr) - 
     \log \Bigl( e^{\kappa(y)} - t \int_Y\psi\,dm\Bigr) \Bigr|
    & \leq \frac{|\kappa(x) - \kappa(y)|}{1-t e^{|\psi|_{\eta,\ell}}}
     \leq \frac{|\psi|_{\eta,\ell} \, d(x,y)^\eta}{1-t e^{|\psi|_{\eta,\ell}}}.
  \end{align*}
  This completes the proof.
\end{proof}

Fix constants $R>0$ and $\xi \in (0,e^{-R})$, such that 
$ R (1-\xi e^R) \geq K + \lambda^{-\eta} R$.
(For example, choose $R$ and $\xi$ as in Remark~\ref{rmk-main}.)

\begin{prop} 
  \label{prop-R}
Let $\psi:Y\to(0,\infty)$ with
  $|\psi|_{\eta,\ell}\le R$. Then $|P_m \psi|_{\eta,\ell}\le R$.
\end{prop}

\begin{proof}
By Proposition~\ref{prop-dist}, $|P_m\psi|_{\eta,\ell}\le K+\lambda^{-\eta} R\le R$.
\end{proof}

\begin{lemma} 
  \label{lem-coup}
  Let $\psi_1,\,\psi_2:Y\to(0,\infty)$ with 
$|\psi_1|_{\eta,\ell}\le R$, $|\psi_2|_{\eta,\ell}\le R$, and
  $\int_Y \psi_1 \, dm = \int_Y \psi_2 \, dm$.
  Let $\psi_j' = P_m \psi_j - \xi \int_Y\psi_j\,dm$ for $j=1,2$.
  Then
  \begin{itemize}
    \item[(a)] $|\psi_j'|_{\eta,\ell}\le R$ for $j=1,2$,
    \item[(b)] $P_m \psi_1 - P_m \psi_2 = \psi_1' - \psi_2'$,
    \item[(c)] $\int_Y \psi_1' \, dm = \int_Y \psi_2' \, dm = (1-\xi) \int_Y\psi_1\,dm$.
  \end{itemize}
\end{lemma}

\begin{proof}
  By Propositions~\ref{prop-dist} and~\ref{prop-diet},
  \[
   |\psi_j'|_{\eta,\ell}= \Bigl| P_m \psi_j - \xi \int_Y \psi_j \, dm \Bigr|_{\eta,\ell}
    \leq
    \frac{|P_m \psi_j|_{\eta,\ell}}{1-\xi e^{|P_m \psi_j|_{\eta,\ell}}}
    \leq 
    \frac{K+\lambda^{-\eta} R}{1-\xi e^R} \leq R,
  \]
proving part~(a). Parts~(b) and~(c) are immediate.
\end{proof}

Now we are ready to prove Theorem~\ref{thm-UE}
  taking $C = 4e^R(1+R)$ and $\gamma = 1-\xi$.

\begin{pfof}{Theorem~\ref{thm-UE}}
  Assume first that $|\phi|_\eta \leq R$. Later we remove this restriction.

Since $\int_Y\phi\,dm=0$, there exists $x,y\in Y$ such that
$\phi(x)\le0\le\phi(y)$.  Hence it follows from the assumption $|\phi|_\eta\le R$ that $|\phi|_\infty\le R$.
  
  Write $\phi=\psi_0^+-\psi_0^-$, where $\psi_0^+ = 1+\max\{0, \phi\}$ and
  $\psi_0^- = 1-\min\{0, \phi\}$.
  Then $\psi_0^\pm:Y\to[1,\infty)$ and $\int_Y\psi_0^+\,d\mu=\int_Y\psi_0^-\,d\mu
  \leq 1+ |\phi|_\infty\le 1 + R$.
  For $x,y \in Y$,
  \[
    \left| \log \psi_0^\pm (x) - \log \psi_0^\pm(y) \right|
    \leq
    \left| \psi_0^\pm (x) - \psi_0^\pm(y) \right|
    \leq
    \left| \phi (x) - \phi(y) \right|,
  \]
  so $|\psi_0^\pm|_{\eta,\ell} \leq |\phi|_\eta \leq R$. 

  Define 
  \begin{align*}
    \psi_{n+1}^\pm &= P_m \psi_n^\pm - \xi \int_Y\psi_n^\pm\,dm,\quad n\ge0.
  \end{align*}
  By Lemma~\ref{lem-coup}(a), $|\psi_n^\pm|_{\eta,\ell}\le R$ for all $n\ge0$.
  By Lemma~\ref{lem-coup}(b,c),
  \begin{align}
    \label{eq-sxg}
    P_m^n \phi = P_m^n \psi_0^+ - P_m^n \psi_0^- 
    = \psi_n^+ - \psi_n^-,
  \end{align}
  and $\int_Y \psi_n^\pm \, dm = \gamma^n \int_Y \psi_0^\pm \, dm\le (1+R)\gamma^n$.
  By~\eqref{eq-bound}, 
  \begin{align}
    \label{eq-unap}
    \psi_n^\pm \leq e^R\! \int_Y \psi_n^\pm \, dm
    \leq e^R (1+R) \gamma^n.
\end{align}
  
  Next, we recall the inequality 
\begin{align} \label{eq-ineq}
|a - b| \leq \max\{a,b\} \, |\log a-\log b|,\;\text{for all $a,b>0$}.
\end{align}
  By~\eqref{eq-ineq} and the definition of $|\psi|_{\eta,\ell}$, for $x,y \in Y$,
  \begin{align*}
    \bigl|\psi_n^\pm(x) - \psi_n^\pm(y) \bigr|
    & \leq \max(\psi_n^\pm(x), \psi_n^\pm(y)) \,
      \bigl|\log \psi_n^\pm(x) - \log \psi_n^\pm(y)\bigr| \\
      & \leq e^R  (1+R) \gamma^n  |\psi_n^\pm|_{\eta,\ell}\, d(x,y)^\eta
       \leq e^R R (1+R) \gamma^n   d (x,y)^\eta.
  \end{align*}
  Hence, $|\psi_n^\pm|_\eta \leq e^R R (1+R) \gamma^n$.
  By \eqref{eq-sxg}, 
\begin{align} \label{eq-unap2}
| P_m^n\phi|_\eta \leq 2 e^R R (1+R) \gamma^n.
\end{align}

  Finally, we remove the restriction $| \phi|_\eta \leq R$. Note that
  $u=R |\phi|_\eta^{-1}  \phi$ satisfies
  $| u|_\eta \leq R$, and therefore it follows from~\eqref{eq-unap2} that
  \[
    | P_m^n \phi|_\eta = R^{-1}|\phi|_\eta\, | P_m^n u|_\eta
    \leq 2e^R (1+R) \gamma^n\,|\phi|_\eta.
  \]
Also, $\int_Y P_m^n\phi\,dm=0$, so $|P_m^n\phi|_\infty\le |P_m^n\phi|_\eta$.
Hence
  \[
    \|P_m^n \phi\|_\eta \le  2 |P_m^n \phi|_\eta 
    \leq 4e^R (1+R) \gamma^n\,|\phi|_\eta,
  \]
as required.
\end{pfof}

\begin{pfof}{Proposition~\ref{prop-inv}}
We construct an invariant probability measure $\mu\in\cM$ and show that
$|\frac{d\mu}{dm}|_{\eta,\ell}\le R$.

By Proposition~\ref{prop-R}, $|P_m^n 1|_{\eta,\ell}\le R$ for all $n\ge0$.  
In particular, it follows from~\eqref{eq-bound} that 
$|P_m1|_\infty\le e^R$.
By~\eqref{eq-ineq}, 
\[
  | P_m1|_\eta\le |P_m1|_\infty |P_m1|_{\eta,\ell}\le e^R R.
\]
  Also, $\int_Y (P_m1 - 1) \, dm = 0$, so by Theorem~\ref{thm-UE},
  $\|P_m^n (P_m1 - 1)\|_\eta \leq Ce^R R\gamma^n$.
  Hence we can define
  \[
    \rho = \lim_{n \to \infty} P_m^n 1 = 1 + \sum_{n=0}^{\infty} P_m^n (P_m1 - 1)\in C^\eta.
  \]
  It is immediate that $\int_Y \rho \, dm = 1$ and $P_m \rho = \rho$,
so $\rho$ is an invariant density.
Moreover, for $x,y\in Y$,
  \begin{align*}
    |\log\rho(x)-\log\rho(y)|=\lim_{n\to\infty}|\log(P_m^n1)(x)-\log(P_m^n1)(y)|\le Rd(x,y)^\eta,
  \end{align*}
so that $|\rho|_{\eta,\ell}\le R$.
\end{pfof}

\begin{rmk}
In this paper, we have restricted attention to expanding maps $F:Y\to Y$ satisfying the full branch condition $Fa=Y$ for all $a\in\alpha$.
This is a reasonable restriction for situations where the expanding maps are obtained by inducing nonuniformly expanding maps as in~\cite{KKMsub1}.
More generally, the restriction is justified by the family of examples $F_\delta:[0,1]\to[0,1]$ depicted in 
Figure~\ref{fig-example} below.   Note that each map preserves Lebesgue measure and is mixing.   Moreover,
we can take $\lambda=2$ and $K=0$ for all $\delta$.  Nevertheless, 
correlations decay arbitrarily slowly as $\delta\to0$.
(Explicit constants depending on $\delta$ can be computed from~\cite{Zweimuller04}.)

\begin{figure}[h]
\centering
\begin{tikzpicture}
  \newlength{\arrowsize}  
  \pgfarrowsdeclare{biggertip}{biggertip}{  
    \setlength{\arrowsize}{0.4pt}  
    \addtolength{\arrowsize}{.5\pgflinewidth}  
    \pgfarrowsrightextend{0}  
    \pgfarrowsleftextend{-5\arrowsize}  
  }{  
    \setlength{\arrowsize}{0.4pt}  
    \addtolength{\arrowsize}{.5\pgflinewidth}  
    \pgfpathmoveto{\pgfpoint{-5\arrowsize}{4\arrowsize}}  
    \pgfpathlineto{\pgfpointorigin}  
    \pgfpathlineto{\pgfpoint{-5\arrowsize}{-4\arrowsize}}  
    \pgfusepathqstroke  
  }  
  \begin{axis}[
    unit vector ratio*=1 1 1, 
    width=9cm, 
    axis lines = middle,
    axis line style={-biggertip},
    enlargelimits = true,
    xtick={0, 0.25, 0.42, 0.5, 0.58, 0.75, 1.0},
    xticklabels={$0$, $\frac{1}{4}$,, $\frac{1}{2}$,, $\frac{3}{4}$, $1$},
    ytick={0, 0.5, 1.0},
    yticklabels={0.0,$\frac{1}{2}$, $1$}
  ]

  \addplot[mark = none, color=gray, loosely dashed] coordinates {(0.0, 0.5) (1.0, 0.5)};
  \addplot[mark = none, color=gray, loosely dashed] coordinates {(0.0, 1.0) (1.0, 1.0)};
  \addplot[mark = none, color=gray, loosely dashed] coordinates {(0.25, 0.0) (0.25, 1.0)};
  \addplot[mark = none, color=gray, loosely dashed] coordinates {(0.42, 0.0) (0.42, 1.0)};
  \addplot[mark = none, color=gray, loosely dashed] coordinates {(0.5, 0.0) (0.5, 1.0)};
  \addplot[mark = none, color=gray, loosely dashed] coordinates {(0.58, 0.0) (0.58, 1.0)};
  \addplot[mark = none, color=gray, loosely dashed] coordinates {(0.75, 0.0) (0.75, 1.0)};
  \addplot[mark = none, color=gray, loosely dashed] coordinates {(1.0, 0.0) (1.0, 1.0)};
 
  \draw[biggertip-biggertip] (axis cs:0.42,0.2) -- (axis cs:0.58,0.2);
  \node[fill=white,above] at (axis cs:0.5,0.2) {$\delta$};

  \addplot[mark = none, ultra thick] coordinates {(0, 0) (0.25, 0.5)};
  \addplot[mark = none, ultra thick] coordinates {(0.25, 0) (0.42, 0.5)};
  \addplot[mark = none, ultra thick] coordinates {(0.42, 1.0) (0.58, 0.0)};
  \addplot[mark = none, ultra thick] coordinates {(0.58, 0.5) (0.75, 1.0)};
  \addplot[mark = none, ultra thick] coordinates {(0.75, 0.5) (1.0, 1.0)};
 
  \node[below left] at (axis cs:0,0) {$0$};
\end{axis}
\end{tikzpicture}
\caption{A family of uniformly expanding maps $F_\delta:[0,1]\to[0,1]$ with $\lambda=2$ and $K=0$ but with arbitrarily slow decay of correlations.} \label{fig-example}
\end{figure}
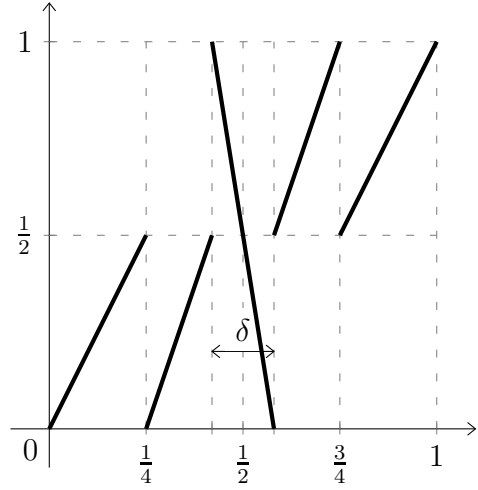
\end{rmk}

\section{Proof for nonuniformly expanding maps}
\label{sec-NUE}

In this section, we prove Theorems~\ref{thm-NUE} and~\ref{thm-NUE2}.
The coupling technique from probability theory, on which our proofs are based, was introduced to dynamical systems by Young~\cite{Young99}, and has since been used in various ways by numerous authors, including~\cite{BL,CD,Zweimuller04}. Our proof is in many ways similar to those in the above works, but is also different: to obtain explicit control on various constants,  we developed a new (to the best of our knowledge) construction of coupling and the method to apply it.

\subsection{Outline of the proof}
\label{sec-outline}

Let \(\Delta_\ell = \{(y, k) \in \Delta : k=\ell \}\) denote the \(\ell\)-th level of the tower. 
Our strategy is to construct a countable probability space \((W, \bP)\) and a random
variable \(h : W \to \N\) such that every sufficiently regular
observable \(\psi : \Delta \to [0, \infty)\) with
\(\int_\Delta \psi \, dm_\Delta = 1\) can be decomposed into
a series
  $\psi = \sum_{w \in W} \psi_w$
where \(\psi_w : \Delta \to [0, \infty) \) are such that
\(\int_\Delta \psi_w \, dm_\Delta = \bP(w)\) and 
\(L^{h(w)} \psi_w = \bP(w) \bar{\tau} 1_{\Delta_0}\).

Now let \(\phi : \Delta \to \R\) and  suppose that  
\(L^N\phi = C(\psi - \psi')\) where \(\psi\) and \(\psi'\)
can be decomposed as above and $C>0$, $N\in\N$ are constants.
We have
\(L^{h(w)} \psi_w = L^{h(w)} \psi'_w\), and so
\(
  L^n (\psi_w - \psi'_w)
  = 0
\)
whenever \(n \geq h(w)\).
Therefore
\begin{align*}
  \int_\Delta |L^{N+n} \phi| \, dm_\Delta
  & \leq C\sum_{w \in W : h(w)>n} \int_\Delta (L^n \psi_w + L^n \psi'_w) \, dm_\Delta
  \\ & = C\sum_{w \in W : h(w)>n} \int_\Delta (\psi_w + \psi'_w) \, dm_\Delta
  = 2C \bP(h > n).
\end{align*}
In this way,  decay rates for \(L^n \phi\) reduce to tail estimates for $h$.

\subsection{Recurrence to \(\Delta_0\)}

Given $\psi : \Delta \to [0,\infty)$, define
\[
  |\psi|_{\eta,\ell} = \sup_{n\geq 0} \sup_{(y,n)\neq (y',n) \in \Delta_n} 
  \frac{|\log \psi(y,n)- \log \psi(y',n)|}{d(y,y')^\eta},
\]
where \(\log 0 = -\infty\) and \(\log 0 - \log 0 = 0 \).

As in Section~\ref{sec-UE}, we
fix constants $R>0$ and $\xi \in (0,e^{-R})$, such that
$ R (1-\xi e^R) \geq K + \lambda^{-\eta} R$.
(For example, choose $R$ and $\xi$ as in Remark~\ref{rmk-main}.)
Using notation from Section~\ref{sec-UE},
$(L\phi)(y,\ell)=\begin{cases} \phi(y,\ell-1) & \ell\ge1
\\ \sum_{a\in\alpha}\zeta(y_a)\phi(y_a,\tau(a)-1) & \ell=0\end{cases}$.

\begin{prop}
  \label{prop-duude}
  Let $\psi:\Delta\to[0,\infty)$ with \(|\psi|_{\eta, \ell} \leq R\). Then
  \begin{enumerate}[label=(\alph*)]
    \item
      \( \BIG
        e^{-R} \bar{\tau} \int_{\Delta_0} \psi \, dm_\Delta  
        \leq \psi \, 1_{\Delta_0} 
        \leq e^{ R} \bar{\tau} \int_{\Delta_0} \psi \, dm_\Delta .
      \)
    \item \(|L \psi|_{\eta, \ell} \leq R\).
    \item If \(t \in [0, \xi]\), then
      \(
        \psi' =
        L \psi - t \, \bar{\tau} \int_{\Delta_0} L \psi \, dm_\Delta \, 1_{\Delta_0}
      \)
      is nonnegative and \(|\psi'|_{\eta, \ell} \leq R\).
  \end{enumerate}
\end{prop}

\begin{proof}
(a) This is the counterpart of~\eqref{eq-bound}.  

\noindent(b)  Let $(y,\ell),(y',\ell)\in\Delta_\ell$.
If $\ell\ge1$, then it is immediate that
$|\log (L\psi)(y,\ell)-\log (L\psi)(y',\ell)| \le R d(y,y')^\eta$.
The same calculation as in Proposition~\ref{prop-dist} shows that
\[
|\log (L\psi)(y,0)-\log (L\psi)(y,0)|\le 
(K+\lambda^{-\eta}R)d(y,y')^\eta
\le Rd(y,y')^\eta.
\]
\noindent(c)  It follows from (b) that $|L\psi|_{\eta,\ell}\le R$. Hence, by (a),
$\psi'\ge0$.  As in part~(b), it is immediate that 
$|\log \psi'(y,\ell)-\log\psi'(y',\ell)|\le Rd(y,y')^\eta$ for $\ell\ge1$.
Also, $\psi'(y,0)=\chi(y)-t\int_Y\chi dm$ where $\chi:Y\to[0,\infty)$
is given by $\chi(y)=(L\psi)(y,0)$, so it follows from
Proposition~\ref{prop-diet} that 
$|\log \psi'(y,0)-\log\psi'(y',0)|\le 
(K+\lambda^{-\eta}R)(1-te^R)^{-1}d(y,y')^\eta\le Rd(y,y')^\eta$.~
\end{proof}

Define $N=N_1+N_2$ where
\begin{align*}
  & N_1 = \max \{I_k^2\}, \quad
  N_2  = \min \bigl\{ n \geq 1 : m_\Delta (\cup_{\ell \geq n} \Delta_\ell) 
    \leq {\SMALL\frac12}e^{-R} \bar{\tau}^{-1} \bigr\}.
\end{align*}
Let \(\cA\) be the set of observables \(\psi : \Delta \to [0,\infty)\)
such that 
\(|\psi|_\infty \leq e^R \bar{\tau} \int_\Delta \psi \, dm_\Delta\) and
\(|\psi|_{\eta, \ell} \leq R\).
Define \(\cB = L^N \cA\).

\begin{cor} \label{cor-A}
(a) If \(\psi : \Delta \to [0,\infty)\) is supported on
  \(\Delta_0\), and \(|\psi|_{\eta, \ell} \leq R\), then
  \(\psi \in \cA\). 

 \noindent (b) If \(\psi, \psi' \in \cA\) (or \(\cB\)) and \(t \geq 0\), then
  \(L \psi\),
 \(\psi + \psi'\) and \(t \psi\) belong in \(\cA\)
  (or \(\cB\)).  In particular, $\cB\subset\cA$.
\end{cor}

\begin{proof}
Part~(a) follows from Proposition~\ref{prop-duude}(a).
    Next, let \(\psi \in \cA\).  We show that $L\psi\in\cA$; the remaining statements in part~(b) are immediate.
    By Proposition~\ref{prop-duude}(b),
    \(|L \psi|_{\eta, \ell} \leq R\). Also, using the definition
    of \(\cA\) and Proposition~\ref{prop-duude}(a),
    \begin{align*}
      |1_{\Delta \setminus \Delta_0} L \psi |_{\infty}
      &\leq |\psi|_\infty
      \leq e^R \bar{\tau} \int_\Delta \psi \, dm_\Delta
      = e^R \bar{\tau} \int_\Delta L \psi \, dm_\Delta
      , \quad \text{and}
      \\
      |1_{\Delta_0} L \psi|_\infty
      &\leq e^R \bar{\tau} \int_{\Delta_0} L \psi \, dm_\Delta
      \leq e^R \bar{\tau} \int_\Delta L \psi \, dm_\Delta.
    \end{align*}
    Hence 
    \(
      |L \psi|_\infty
      \leq e^R \bar{\tau} \int_\Delta L \psi \, dm_\Delta
    \), so \(L \psi \in \cA\).  
\end{proof}

\begin{prop}
  \label{prop-lllol}
  If \(\psi \in \cA\), then 
    $\max_{0\le j\le N_2}\int_{\Delta_0} L^j \psi \, dm_\Delta 
    \geq \frac{1}{2} e^{-R} \bar{\tau}^{-1}
    \int_\Delta \psi \, dm_\Delta$.
\end{prop}

\begin{proof}
It follows from the definition of $N_2$ and $\cA$ that
$m_\Delta(\cup_{\ell=N_2+1}^\infty \Delta_\ell)|\psi|_\infty\le \frac12\int_\Delta \psi\,dm_\Delta$.   Hence
\begin{align*}
\int_\Delta \psi\,dm_\Delta
& =\int_{\bigcup_{\ell=0}^{N_2}\Delta_\ell} L^{N_2}\psi\,dm_\Delta
+\int_{\bigcup_{\ell= N_2+1}^\infty\Delta_\ell} L^{N_2}\psi\,dm_\Delta
\\ & \le \int_{\bigcup_{\ell=0}^{N_2}\Delta_\ell} L^{N_2}\psi\,dm_\Delta
+\frac12 \int_\Delta \psi\,dm_\Delta,
\end{align*}
so $\int_{\bigcup_{\ell=0}^{N_2} \Delta_\ell} L^{N_2} \psi \, dm_\Delta
    \geq \frac{1}{2}\int_\Delta \psi \, dm_\Delta$.

  Next, if \(\ell \leq N_2\) then \((L^{N_2} \psi)(y, \ell) = (L^{N_2-\ell} \psi)(y, 0)\), and so
$\int_{\Delta_\ell}L^{N_2}\psi\,dm_\Delta\le m_\Delta(\Delta_\ell)|L^{N_2-\ell}\psi\,1_{\Delta_0}|_\infty
\le m_\Delta(\Delta_\ell)\max_{0\le j\le N_2}|L^j\psi\,1_{\Delta_0}|_\infty$.
  Hence, by Proposition~\ref{prop-duude}(a,b),
  \[
    \int_{\bigcup_{\ell=0}^{N_2} \Delta_\ell} L^{N_2} \psi \, dm_\Delta
    \leq \max_{0 \leq j \leq N_2} |L^j \psi\,1_{\Delta_0}|_\infty
    \leq e^R \bar{\tau} \max_{0 \leq j \leq N_2} \int_{\Delta_0} L^j \psi \, dm_\Delta.
  \]
  The result follows.
\end{proof}

\begin{prop}
  \label{prop-kkkok}
  If \(|\psi|_{\eta, \ell} \leq R\), then 
  \(
    \int_{\Delta_0} L^n \psi \, dm_\Delta
    \geq (e^{-R} \delta)^n \int_{\Delta_0} \psi \, dm_\Delta
  \)
  for all \(n \geq N_1\).
\end{prop}

\begin{proof}
  By Proposition~\ref{prop-duude}(a), 
  \(\inf_{\Delta_0} \psi \geq e^{-R} \bar{\tau} \int_{\Delta_0} \psi \, dm_\Delta\).
  By our assumptions,
  \(m_\Delta(\{ x \in \Delta_0 : f^{I_k} x \in \Delta_0 \}) \geq \delta/\bar{\tau}\)
  for every \(I_k\). Hence
  \begin{align*}
    \int_{\Delta_0} L^{I_k} \psi \, dm_\Delta
    & = \int_{\Delta} \psi \,\, 1_{\Delta_0} \circ f^{I_k} \, dm_\Delta
    \geq \int_{\Delta_0} \psi \,\, 1_{\Delta_0} \circ f^{I_k} \, dm_\Delta
    \\ & \geq \inf_{\Delta_0} \psi \,\, m_\Delta(\{ x \in \Delta_0 : f^{I_k} x \in \Delta_0 \})
    \geq e^{-R} \delta \int_{\Delta_0} \psi \, dm_\Delta.
  \end{align*}
  By \cite{Selmer77},
  every \(n \geq N_1\) can be written as
  \(n = \sum_k n_k I_k\),
  where \(n_k\) are nonnegative integers.
  By Proposition~\ref{prop-duude}(b), it follows inductively that 
  \[
    \int_{\Delta_0} L^n \psi \, dm_\Delta
    \geq (e^{-R} \delta)^{\sum_k n_k} \int_{\Delta_0} \psi \, dm_\Delta
    \geq (e^{-R} \delta)^n \, \int_{\Delta_0} \psi \, dm_\Delta,
  \]
as required.
\end{proof}

\begin{lemma} \label{lem-recur}
  If \(\psi \in \cB\), then
  \(
    \int_{\Delta_0} \psi \, dm_\Delta 
    \geq \eps \int_\Delta \psi \, dm_\Delta,
  \)
where $\eps = {\SMALL\frac12} e^{-R} \bar{\tau}^{-1} (e^{-R} \delta)^N$.
\end{lemma}

\begin{proof}
  By definition of \(\cB\), there exists \(\psi' \in \cA\) such that
  \(L^{N_1+N_2} \psi' = \psi\). By Proposition~\ref{prop-lllol}, there exists
  \(j \leq N_2\) such that 
  \(\int_{\Delta_0} L^j \psi' \, dm_\Delta 
  \geq \frac{1}{2} e^{-R} \bar{\tau}^{-1} \int_\Delta \psi' \, dm_\Delta\).
  By Proposition~\ref{prop-kkkok} (taking $n=N_1+N_2-j\ge N_1$),
  \begin{align*}
    \int_{\Delta_0} \psi \, dm_\Delta 
    & = \int_{\Delta_0}L^{N_1+N_2}\psi'\,dm_\Delta 
\geq (e^{-R} \delta)^{N_1+N_2-j} \int_{\Delta_0} L^j \psi' \, dm_\Delta
    \\ & \geq \frac{1}{2} e^{-R} \bar{\tau}^{-1} (e^{-R} \delta)^{N_1+N_2}
    \int_\Delta \psi' \, dm_\Delta = \eps
    \int_\Delta \psi \, dm_\Delta,
  \end{align*}
as required.
\end{proof}

\subsection{Decomposition in $\cB$}

Next, we introduce constants $p_n,\,t_n\in[0,1]$,
\begin{align*}
  & 
t_1 = 1-\eps, \quad 
   t_n = \min \{t_1,\, e^R\bar{\tau} m_\Delta(\cup_{\ell=n}^\infty \Delta_\ell)\}, \,\,n \geq 2, \\ &
     p_{-1}=\xi\eps, \quad p_0=(1-\xi)\eps, \quad p_n = t_n - t_{n+1}, \,\, n \geq 1.
\end{align*}
The monotonicity of the sequence $t_n$ ensures that $p_n\ge0$ for all $n$.
Note that \( \sum_{n=-1}^\infty p_n = 1\).

Let \(E_0 = \Delta_0\) and
\(E_k = \{ (y, \ell) \in \Delta : \ell=\tau(y)-k,\,\ell\ge1\}\)
for \(k \geq 1\). Then $\{E_0,E_1,\ldots\}$ defines a partition of $\Delta$ and 
\(m_\Delta (E_k) = m_\Delta(\Delta_k)\) for all \(k\).

\begin{prop}
  \label{prop-E}
  If \(\psi \in \cB\) with $\int_\Delta \psi \, dm_\Delta=1$, then
  \(
    \int_{\bigcup_{\ell=n}^\infty E_\ell} \psi \, dm_\Delta
    \leq t_n,
  \)
for $n\ge1$.
\end{prop}

\begin{proof}
  By Lemma~\ref{lem-recur}, 
    $\int_{\bigcup_{\ell=n}^\infty E_\ell} \psi \, dm_\Delta
    \le \int_{\bigcup_{\ell=1}^\infty E_\ell} \psi \, dm_\Delta
    \leq 1-\eps=t_1$ for all $n\ge1$.

By definition of $\cB$,
for \(n \geq 2\) we have in addition that
  \(
    \int_{\bigcup_{\ell=n}^\infty E_\ell} \psi \, dm_\Delta
    \leq m_\Delta(\cup_{\ell=n}^\infty \Delta_\ell) |\psi|_\infty
    \leq e^R\bar\tau m_\Delta(\cup_{\ell=n}^\infty \Delta_\ell).
  \)
The result follows by definition of~$t_n$.~
\end{proof}

\begin{prop}  \label{prop-s}
Let $p_j$, $q_j\in[0,\infty)$ be sequences such that
\(\sum_{j=0}^{\infty} p_j = \sum_{j=0}^{\infty} q_j < \infty\)
and $\sum_{j=0}^k q_j\ge \sum_{j=0}^k p_j$ for all $k\ge0$.
Then there exist $s_{k,j}\in[0,1]$, $0\le j\le k$, such that
$\sum_{j=0}^k s_{k,j}q_j=p_k$ for all $k\ge0$ and
$\sum_{k=j}^\infty s_{k,j}=1$ for all $j\ge0$.
\end{prop}

\begin{proof}
  We assume that \(q_j>0\) for all $j$;
  otherwise set \(s_{k,j} = \delta_{k,j}\) for $k\le j$ whenever
  \(q_j = 0\).

  For \(k=0\), choose \(s_{0,0} = p_0 / q_0\).
  Next let \(k\ge1\), and suppose inductively
  that \(s_{k',j}\) have been constructed for \(0\le j\le k' \le  k-1\),
  such that $\sum_{j=0}^{k'} s_{k',j}q_j=p_{k'}$ for $k'\le k-1$ 
  and $\sum_{k'=j}^{k-1} s_{k',j}\le1$ for $j\le k-1$.
  
  Define $s_{k,0}, s_{k,1}, \ldots, s_{k,k}\in[0,1]$ (in this order) by
  \[
    s_{k,j} = \min\Bigl\{ 1-\sum_{k' = j}^{k-1} s_{k',j}, \,\,
    \frac{p_k - \sum_{j'=0}^{j-1} s_{k,j'} q_{j'}}{q_j} \Bigr\}, \quad
  j = 0, 1, \ldots, k.
  \]
  By construction,
  \(\sum_{j=0}^{k} s_{k,j} q_j \leq p_k\).
  If \(\sum_{j=0}^{k} s_{k,j} q_j < p_k\), then necessarily
  \mbox{\(\sum_{k'=j}^{k} s_{k',j} = 1\)} for all
  \(j\le k\), and so
  \[
    \sum_{k'=0}^{k} q_{k'}
    = \sum_{k'=0}^{k} \sum_{j=0}^{k'} s_{k',j} q_j
    = \sum_{j=0}^{k} s_{k,j} q_j + \sum_{k'=0}^{k-1} p_{k'}
    < \sum_{k'=0}^{k} p_{k'},
  \]
  which is a contradiction. Hence \(\sum_{j=0}^{k} s_{k,j} q_j = p_k\).
  
  By the above construction, 
  \(\sum_{j=0}^{k} s_{k,j} q_j = p_k\) for $k\ge0$ and
  \(\sum_{k=j}^{\infty} s_{k,j} \leq 1\) for $j\ge0$.
  Since also
  \(\sum_{j=0}^{\infty} p_j = \sum_{j=0}^{\infty} q_j < \infty\),
  we conclude that
  \(\sum_{k=j}^{\infty} s_{k,j} = 1\).
\end{proof}

\begin{lemma}
  \label{lem-W}
  Let \(\psi : \Delta \to [0, \infty)\) be such that
  \(L^n \psi \in \cB\) for some \(n \geq 0\).
  Then 
    $\psi =  \sum_{k=-1}^{\infty} \psi_k$,
  where \(\psi_k : \Delta \to [0,\infty)\) are such that
  \begin{itemize}
    \item[(i)] \(L^n \psi_{-1} = p_{-1} \bar{\tau} \int_\Delta \psi \, dm_\Delta 1_{\Delta_0}\), \quad (ii)
    \(L^{k+n} \psi_k \in \cA\) for all $k\ge0$,
    \item[(iii)] \(\int_\Delta \psi_k \, dm_\Delta = p_k \int_\Delta \psi \, dm_\Delta\)
      for all \(k \geq -1\).
  \end{itemize}
\end{lemma}

\begin{proof}
  First we consider the case \(n = 0\). 
Suppose without loss that 
  \(\int_\Delta \psi \, dm_\Delta = 1\).
Define $\psi_{-1}=p_{-1}\bar\tau 1_{\Delta_0}$ in accordance with properties~(i) and~(iii).
    
By Lemma~\ref{lem-recur},
$\int_{\Delta_0}\psi\,dm_\Delta\ge\eps$.
Hence $t=p_{-1}/\int_{\Delta_0}\psi\,dm_\Delta=\xi\eps/\int_{\Delta_0}\psi\,dm_\Delta\in[0,\xi]$.
Since $\psi\in \cB\subset L\cA$, it follows from
Proposition~\ref{prop-duude}(c) that
\[
\SMALL \psi'=\psi-t\bar\tau\int_{\Delta_0}\psi\,dm_\Delta\,1_{\Delta_0}
=\psi-p_{-1}\bar\tau 1_{\Delta_0}=\psi-\psi_{-1}
\]
is nonnegative and $|\psi'|_{\eta,\ell}\le R$.
Setting $g_0=\psi'1_{\Delta_0}$, we obtain that
$\psi 1_{\Delta_0}=\psi_{-1}+g_0$ where $g_0$ is nonnegative and 
$|g_0|_{\eta,\ell}\le R$.
By Corollary~\ref{cor-A}(a),
$g_0\in\cA$.

Define \(g_k = \psi 1_{E_k}\) for $k\ge1$. 
Note that \(L^k g_k\) is supported on \(\Delta_0\) and \(|L^k g_k|_{\eta, \ell} \leq R\).  By Corollary~\ref{cor-A}(a), \(L^k g_k \in \cA\). 
  
Now \(\psi = \psi1_{\Delta_0}+\sum_{k=1}^\infty g_k 
=\psi_{-1}+\sum_{k=0}^\infty g_k\). 
By Proposition~\ref{prop-E}, 
  \begin{align*}
    p_{-1}+\sum_{j=0}^k \int_\Delta g_j \, dm_\Delta
    = 1 - \sum_{j=k+1}^\infty \int_\Delta g_j \, dm_\Delta 
    \geq 1-t_{k+1} 
    = \sum_{j={-1}}^k p_j.
  \end{align*}
Setting $q_k=\int_\Delta g_k\,dm_\Delta$, we have that
    $\sum_{j=0}^k q_j\ge \sum_{j=0}^k p_j$ for all $k\ge0$.
Choose  \(s_{k,j} \in [0, 1]\) as in
  Proposition~\ref{prop-s}, and 
  define \(\psi_k : \Delta \to [0,\infty)\), $k\ge0$, by
  \begin{align*}
\SMALL \psi_k=\sum_{j=0}^k s_{k,j}g_j.
  \end{align*}
By construction, condition~(iii) holds for all $k$.
Condition~(ii) is satisfied by Corollary~\ref{cor-A}(b).
Finally, by Proposition~\ref{prop-s},
\begin{align*}
\sum_{k=0}^\infty\psi_k=\sum_{k=0}^\infty\sum_{j=0}^k s_{k,j}g_j
=\sum_{j=0}^\infty\sum_{k=j}^\infty s_{k,j}g_j
=\sum_{j=0}^\infty g_j =\psi-\psi_{-1},
\end{align*}
completing the proof for $n=0$.

Now suppose $L^n\psi\in\cB$ for some $n\ge1$.
Setting $\psi'=L^n\psi$ and applying the result for $n=0$, we can write
$\psi'=\sum_{k=-1}^\infty\psi'_k$ where $\psi'_k$ satisfy properties (i)--(iii).
Define
   \(\psi_k = \bigl( \frac{\psi'_k}{\psi'} \circ f^n \bigr) \, \psi\)
  with the convention that \(0/0 = 0\).
  Then $L^n\psi_k=\frac{\psi'_k}{\psi'}L^n\psi=\psi'_k$, so
properties (i)--(iii) are passed down from  $\psi'_k$ to $\psi_k$.
Also $\sum_{k=-1}^\infty \psi_k=\big(\frac{\psi'}{\psi'}\circ f^n\big)\psi=\psi$.~
\end{proof}

Let \(W\) be the countable set of all finite words in the alphabet 
$\N=\{0,1,2,\ldots\}$
including the zero length word, 
and let \(W_k\) be the subset consisting of words of length \(k\). 
Let \(\bP\) be the  probability measure
on \(W\) given for \(w = w_1 \cdots w_k \in W\) by
\(\bP(w) = p_{-1} p_{w_1} \cdots p_{w_k}\).
Define $h:W\to \N$ by
$h(w)=\Sigma w+N|w|$, where
$\Sigma w=w_1+\dots+ w_k$ and $|w|=k$ for $w=w_1\cdots w_k$.

\begin{prop}
  \label{prop-W}
  Let \(\psi \in \cB\) with \(\int_\Delta \psi \, dm_\Delta = 1\).
  Then
    $\psi = \sum_{w \in W} \psi_w$,
  where \(\psi_w : \Delta \to [0, \infty)\) are
  such that \(\int_\Delta \psi_w \, dm_\Delta = \bP(w)\)
  and \(L^{h(w)} \psi_w = \bP(w) \bar{\tau} 1_{\Delta_0}\).
\end{prop}

\begin{proof}
  Write \(\psi = \sum_{k=-1}^\infty \psi_k\) as in Lemma~\ref{lem-W}
(with $n=0$).
By properties~(iii) and~(i),
  \(\int_\Delta \psi_k \, dm_\Delta = p_k\) for all $k\ge-1$,
and $\psi_{-1}=p_{-1}\bar\tau 1_{\Delta_0}$.

  Also \(L^{k+N} \psi_k \in \cB\) by property~(ii), allowing us to
  apply Lemma~\ref{lem-W} to each $\psi_k$ (with $n=k+N$), yielding
\[
\psi=\psi_{-1}+\sum_{k=0}^\infty\psi_k
=\psi_{-1}+\sum_{k=0}^\infty\Big(\psi_{-1,k}+\sum_{j=0}^\infty\psi_{j,k}\Big),
\]
where 
\begin{align*}
\int_\Delta\psi_{j,k}\,dm_\Delta & =
p_j \int_\Delta \psi_k\,dm_\Delta =p_jp_k, \\
L^{k+N}\psi_{-1,k} & =p_{-1}\bar\tau\int_\Delta \psi_k\,dm_\Delta\,1_{\Delta_0}=
p_{-1}p_k\bar\tau 1_{\Delta_0}.
\end{align*}

At the next step,
\begin{align*}
\psi
& =\psi_{-1}+\sum_{k=0}^\infty\psi_{-1,k}+\sum_{j,k=0}^\infty\Big(\psi_{-1,j,k}+
\sum_{i=0}^\infty\psi_{i,j,k}\Big),
\\ & =\psi_{-1}+\sum_{w\in W_1}\psi_{-1,w}+
\sum_{w\in W_2}\psi_{-1,w}+\sum_{i,j,k=0}^\infty \psi_{i,j,k},
\end{align*}
where 
\[
\int_\Delta \psi_{i,j,k}=p_ip_jp_k,  \quad
L^{j+k+2N}\psi_{-1,j,k} = p_{-1}p_jp_k\bar\tau 1_{\Delta_0}.
\]
In particular, for the terms $\psi_{-1,w}$ with $w\in W_0\cup W_1\cup W_2$,
we have the required properties
$\int_{\Delta}\psi_{-1w}\,dm_\Delta=\bP(w)$
and $L^{h(w)}\psi_{-1,w}=\bP(w)\bar\tau 1_{\Delta_0}$.

In this way we obtain $\psi=\sum_{w\in W}\psi_{-1w}$ where
$\int_{\Delta}\psi_{-1w}\,dm_\Delta=\bP(w)$
and $L^{h(w)}\psi_{-1,w}=\bP(w)\bar\tau 1_{\Delta_0}$.
\end{proof}

\subsection{Proof of Theorems~\ref{thm-NUE} and~\ref{thm-NUE2}}

Let \(\phi : \Delta \to \R\) be an observable as in Theorem~\ref{thm-NUE},
i.e.\ \(\|\phi\|_{\eta} < \infty \) and \(\int_\Delta \phi \, dm_\Delta = 0\).
Define \(\tilde\psi, \tilde\psi' : \Delta \to [0,\infty)\) by
\[
  \tilde\psi = 1 + \frac{\phi}{\|\phi\|_\eta (1+R^{-1})}
  , \qquad
  \tilde\psi' \equiv 1.
\]
Then \(L^n\phi = \|\phi\|_\eta (1+R^{-1}) (\psi - \psi')\),
where $\psi=L^N\tilde\psi$,
$\psi'=L^N\tilde\psi'$.

Now $\int_\Delta\psi\,dm_\Delta=\int_\Delta\psi'\,dm_\Delta=1$.
  Next, 
  \[
    |\tilde\psi|_\infty \leq 1+ \frac{1}{1+R^{-1}}
    \leq 1+R
    \leq e^R \le \bar\tau e^R.
  \]
Also, $|\tilde\psi|\ge 1-(1+R^{-1})^{-1}=(1+R)^{-1}$ and
$|\tilde\psi|_\eta\le(1+R^{-1})^{-1}$, so
  for \(x,y \in \Delta\),
  \[
    |\log \tilde\psi (x) - \log \psi (y)|
    \leq |\tilde\psi^{-1}|_\infty \, |\tilde\psi(x) - \tilde\psi(y)|
    \leq \frac{R+1}{1+R^{-1}}d_\Delta(x,y)
    = Rd_\Delta(x,y).
  \]
  Thus $|\tilde\psi|_{\eta, \ell}\leq R$.
We have shown that $\tilde\psi\in\cA$, and hence $\psi\in\cB$.
Clearly, $\psi'\in\cB$.

We have shown that $\psi$ and $\psi'$ satisfy the hypotheses of
Proposition~\ref{prop-W} and hence admit the decompositions 
given in the conclusion of Proposition~\ref{prop-W}.
We are therefore in the situation described in Subsection~\ref{sec-outline} 
(with $C=\|\phi\|_\eta(1+R^{-1})$, and the argument there shows that
  \[
    \int_\Delta |L^{N+n} \phi| \, dm_\Delta
    \leq 2 \|\phi\|_\eta (1+R^{-1}) \bP(h>n).
  \]
  To prove Theorem~\ref{thm-NUE}, it remains to
  estimate the decay of \(\bP(h>n)\).

Recall that \(W_k\) is the subset of \(W\) consisting of words of length \(k\). Then
\(\bP(W_k) = (1-p_{-1})^k p_{-1}\).
Elements of \(W_k\) have the form \(w_1 \cdots w_k\)
where \(w_1, \ldots, w_k\) can be regarded as independent identically
distributed random variables, drawn from $\N$ with distribution
\(\bP(w_1 = n) = p_n/(1-p_{-1})\).
Also, $\bP(|w|\ge n)=(1-p_{-1})^n$.

\subsubsection*{Polynomial tails}

\begin{prop}
  \label{prop-poly}
Suppose that there exists $C_\tau>0$ and $\beta>1$ such that
\(m(\tau \geq n) \leq C_\tau n^{-\beta}\) for \(n \geq 1\).  
  
Then \( \bP(h \geq n) \leq C n^{-(\beta-1)} \) for \(n \geq 1\),
  where \(C\) depends continuously
  on \(C_\tau\), \(\beta\), $R$, \(N\) and \(p_{-1}\).
\end{prop}

\begin{proof}
  Let $\tilde t_n=\bar\tau e^R m_\Delta(\bigcup_{\ell=n}^\infty \Delta_\ell)$.
Then 
  \[
    p_n = t_n-t_{n+1}\le \tilde t_n-\tilde t_{n+1}
    = \bar{\tau} e^{R} m_\Delta(\Delta_n)
    = e^R m(\tau \geq n)
    \leq e^R C_\tau n^{-\beta}.
  \]
  Using the inequality
  \(
    \sum_{j \geq n} j^{-\beta} \leq n^{-\beta}+\int_n^\infty x^{-\beta}\,dx\leq \beta n^{-(\beta-1)}/(\beta-1)
  \),
  we obtain
  \[
    \bP (w_1 \geq n) 
    = (1-p_{-1})^{-1} \sum_{j \geq n} p_j
    \leq C_\tau e^R (1-p_{-1})^{-1} \frac{\beta n^{-(\beta-1)}}{\beta-1}=
C_1 n^{-(\beta-1)},
  \]
where \(C_1 =  C_\tau e^R (1-p_{-1})^{-1} \beta (\beta-1)^{-1}\).
  It follows that for $w\in W$, $k\ge1$,
  \begin{align*}
    \bP(\Sigma w\ge n\,|\,w\in W_k) & =\bP(w_1+\dots+ w_k\ge n)
\\ & \le \sum_{j=1}^k \bP(w_j\ge n/k)  = k\bP(w_1\ge n/k) \le C_1 k^\beta n^{-(\beta-1)}.
\end{align*}
Hence
\begin{align*}
\bP(\Sigma w\ge n) & =\sum_{k=1}^\infty \bP(\Sigma w\ge n\,|\,w\in W_k)
\, \bP(W_k)
\\ & \le C_1n^{-(\beta-1)}\sum_{k=1}^\infty k^\beta (1-p_{-1})^kp_{-1}
=C_1' n^{-(\beta-1)},
\end{align*}
    where $C_1' = C_1 p_{-1} \sum_{k=1}^\infty k^\beta (1-p_{-1})^k$.
Finally, 
\begin{align*}
 \bP(h(w)\ge n)  & =\bP(\Sigma w+N|w|\ge n)
\\ & \le \bP(\Sigma w\ge n/2)+\bP(|w|\ge n/(2N))
    \leq C_1' 2^{\beta-1} n^{-(\beta-1)} + (1-p_{-1})^{n/(2N)}.
  \end{align*}
  The result follows.
\end{proof}

\subsubsection*{(Stretched) exponential tails}

\begin{prop}
  \label{prop-aexp}
Let \(X_1, \ldots, X_k\) be nonnegative random variables.
  Suppose that there exist \(\alpha>0\), $\gamma\in(0,1]$, such that
  \[
\bP( X_j \geq t\,|\,X_1=x_1,\dots,X_{j-1}=x_{j-1}) \leq C e^{-\alpha t^\gamma}
\]
  for all \(t \geq 0\), $1\le j\le k$ and $x_1,\dots,x_{j-1}\ge0$. Then for all $\beta\in(0,\alpha/2]$, $t\ge0$,
  \[
    \bP(X_1 + \cdots + X_k \geq t)
    \leq (1+\beta C_1)^k e^{-\beta t^\gamma},
  \]
  where \(C_1\) depends continuously 
  on $C$, \(\gamma\) and \(\alpha\).
\end{prop}

\begin{proof}
  Note that
    $\bE (e^{\beta X_1^\gamma}) 
    = \int_{0}^{\infty} \bP(e^{\beta X_1^\gamma} \geq t) \, dt
    = 1 + \int_{1}^{\infty} \bP(e^{\beta X_1^\gamma} \geq t) \, dt$.
  Making the substitution \(t = e^{\beta s^\gamma}\), we obtain
  \[
    \bE (e^{\beta X_1^\gamma}) 
    = 1 + \beta \gamma \int_{0}^{\infty} s^{\gamma-1} e^{\beta s^\gamma} \bP(X_1 \geq s) \, ds
    \leq 1 + C \beta \gamma \int_{0}^{\infty} s^{\gamma-1} e^{-(\alpha - \beta) s^\gamma} \, ds
    \le  1+\beta C_1,
  \]
where $C_1=C\gamma\int_0^\infty s^{\gamma-1} e^{-\frac12\alpha s^\gamma}\,ds$.
Similarly, $\bE(e^{\beta X_j^\gamma}\,|\,X_1,\dots,X_{j-1})\le 1+\beta C_1$.
  Hence
  \begin{align*} 
    \bE \bigl(e^{\beta (X_1 + \cdots + X_k)^\gamma}\bigr) 
    & \leq \bE \bigl(e^{\beta (X_1^\gamma + \cdots + X_k^\gamma)}\bigr)
=\bE[\bE(e^{\beta(X_1^\gamma+\dots+X_k^\gamma)}\,|\,X_1,\dots,X_{k-1})]
    \\ & =
\bE[e^{\beta(X_1^\gamma+\dots+X_{k-1}^\gamma)}\bE(e^{\beta X_k^\gamma}\,|\,X_1,\dots,X_{k-1})]
\\ & \le  (1+\beta C_1) \bE(e^{\beta(X_1^\gamma+\dots+X_{k-1}^\gamma)})
\leq \dots \leq (1+\beta C_1)^k
    .
  \end{align*}
  The result follows from Markov's inequality.
\end{proof}

\begin{prop}
  \label{prop-stretch}
Suppose that there exist $C_\tau,\,A>0$, $\gamma\in(0,1]$ such that \(m(\tau \geq n) \leq C_\tau e^{-A n^\gamma}\) for \(n \geq 1\).

Then
  \( \bP(h \geq n) \leq C e^{-B n^\gamma}\) for all \(n \geq 1\),
  where \(C>0\) and \(B \in(0, A)\) depend continuously
  on \(C_\tau\), \(A\), \(\gamma\), $R$, \(N\) and \(p_{-1}\).
\end{prop}

\begin{proof}
Following the proof of Proposition~\ref{prop-poly},
    $p_n \leq  e^R m(\tau \geq n) \leq e^R C_\tau e^{-A n^\gamma}$.
Using that $x^q\le (2q)^qe^{x/2}$ for all $x,q>0$,
  \begin{align*}
    \sum_{j \geq n} e^{-A j^\gamma}
    & \leq e^{-A n^\gamma} + \int_{n}^{\infty} e^{-A t^\gamma} \, dt
= e^{-A n^\gamma} + \gamma^{-1}  A^{-1/\gamma} \int_{A n^\gamma}^{\infty} e^{-s} s^{\frac{1}{\gamma} - 1} \, ds
\\ & \le  e^{-A n^\gamma} + C_{A,\gamma} \int_{A n^\gamma}^{\infty} e^{-s/2} \, ds \leq 3C_{A, \gamma}\, e^{-\frac12 A n^\gamma},
  \end{align*}
  where $C_{A,\gamma}\ge1$ is a constant depending continuously on \(A, \gamma\).
  Hence
  \[
    \bP (w_1 \geq n) 
    = (1-p_{-1})^{-1} \sum_{j \geq n} p_j
\leq 3(1-p_{-1})^{-1} e^R C_\tau C_{A,\gamma}
     e^{-\frac12 A n^\gamma}.
  \]

By Proposition~\ref{prop-aexp}, for $B\in(0,\frac14 A]$,
\[
\bP(\Sigma w \geq n\,|\,w\in W_k)=\bP(w_1+\dots+w_k\ge n) \leq (1+BC_1)^k e^{-B n^\gamma},
\]
  where $C_1$ depends continuously on
  \(C_\tau\), $A$, $\gamma$, \(R\), \(p_{-1}\).

Let $r=(1+BC_1)(1-p_{-1})$ and
choose \(B\) small enough that $r<1$.  Then
  \[
    \bP(\Sigma w \geq n)
    = \sum_{k=0}^\infty \bP(\Sigma w \geq n\,|\,w\in W_k) \, \bP(W_k)
    \leq e^{-B n^\gamma} p_{-1} \sum_{k=0}^\infty r^k =C' e^{-B n^\gamma},
  \]
  where \(C' = p_{-1} (1-r)^{-1}\).

Finally, 
\begin{align*}
\bP(h(w)\ge n) & =\bP(\Sigma w+N|w|\ge n)
\\ & \le \bP(\Sigma w\ge n/2)+\bP(|w|\ge n/(2N))
    \leq C' e^{- B n^\gamma/2^\gamma} + (1-p_{-1})^{n/(2N)}.
  \end{align*}
  The result follows.
\end{proof}

\begin{pfof}{Theorem~\ref{thm-NUE2}}
  As in the proof of Theorem~\ref{thm-NUE}, we can write
  \(\phi = C_0 (\psi - \psi')\), where \(C_0 = \|\phi\|_{\eta}(1+R^{-1})\),
  and \(\psi, \psi' \in \cA\) with 
  \(\int_\Delta \psi \, dm_\Delta = \int_\Delta \psi' \, dm_\Delta = 1\). 
By Corollary~\ref{cor-A}(b),
  \(|L^n \psi|_{\eta, \ell}, |L^n \psi'|_{\eta, \ell} \leq R\) and 
  \(|L^n \psi|_\infty, |L^n \psi'|_\infty \leq \bar{\tau} e^R\)
  for all \(n \geq 0\).
  
  Next
  \[
    \bigl| (L^n \psi)(x) - (L^n \psi)(y) \bigr| 
    \leq |L^n \psi|_\infty
    \bigl|\log (L^n \psi)(x) - \log (L^n \psi)(y)\bigr|,
  \]
  so \(|L^n \psi|_\eta \leq \bar{\tau} e^R R \). Similarly,
  \(|L^n \psi'|_\eta \leq \bar{\tau} e^R R \).
Hence
  \[
    \|L^n \phi\|_{\eta}
    \leq C_0 ( \|\psi\|_{\eta} + \|\psi'\|_{\eta} )
    \leq C_0 (2 \bar{\tau} e^R + 2 \bar{\tau} e^R R)
    \leq C_1 \|\phi\|_{\eta},
  \]
  where \( C_1 = 2\bar{\tau} e^R (1+R)(1+R^{-1})\).
  Let \(\tilde{\phi} = \sum_{k=0}^{d-1} L^k \phi\).
Then
  \(
    \|\tilde{\phi}\|_{\eta} 
    \leq C_1 d \|\phi\|_{\eta}.
  \)
  
  For \(r=0, \ldots, {d-1}\), define
  \(
    \Delta(r) = \{ (y, \ell) \in \Delta : \ell \equiv r \bmod{d} \}.
  \)
  Then \(f^d : \Delta(r) \to \Delta(r)\) is a mixing Young tower with data
$\{I_k/d\}$, $\delta$, replacing
the data $\{I_k\}$, $\delta$, for $\Delta$.

  Note that \( \sum_{k=0}^{d-1} 1_{\Delta(r)} \circ f^k \equiv 1\).
  Hence for $r=0,\dots,d-1$,
  \[
    \int_{\Delta(r)} \tilde{\phi} \, dm_\Delta 
    = \sum_{k=0}^{d-1} \int_{\Delta} 1_{\Delta(r)} L^k \phi \, dm_\Delta
    = \sum_{k=0}^{d-1} \int_{\Delta} 1_{\Delta(r)} \circ f^k \, \phi \, dm_\Delta
    = \int_\Delta \phi \, dm_\Delta = 0.
  \]

Thus for each $r=0,\dots,d-1$, we are in the situation of Theorem~\ref{thm-NUE} with $\Delta$, $f$, $\phi$ replaced by $\Delta(r)$, $f^d$, $\tilde\phi$.
In the case of polynomial tails, 
  \begin{align*}
    \int_\Delta \Bigl| \sum_{k=0}^{d-1} L^{nd+k} \phi \Bigr| \, dm_\Delta
   &  = \sum_{r=0}^{d-1} \int_{\Delta(r)} \Bigl| \sum_{k=0}^{d-1} L^{nd+k} \phi \Bigr| \, dm_\Delta
    = \sum_{r=0}^{d-1} \int_{\Delta(r)} 
    |L^{nd} \tilde{\phi}| \, dm_\Delta
\\ & \le C\|\tilde\phi\|_\eta(nd)^{-(\beta-1)}
 \le CC_1d^{-(\beta-1)}\|\phi\|_\eta\, n^{-(\beta-1)},
  \end{align*}
and similarly for the (stretched) exponential case.
\end{pfof}

\section{Proof for nonuniformly hyperbolic transformations}
\label{sec-NUH}

In this section we prove Theorem~\ref{thm-NUH}.

The separation time for $\bar{F}:\bar{Y}\to \bar{Y}$ extends to a separation time on
$\bar{\Delta}$:
define $s((y,\ell),(y',\ell'))=s(y,y')$
if $\ell=\ell'$ and $0$ otherwise.
Recall that we fixed \(\theta \in (0,1)\). Define the metric $d_\theta$ on
\(\bar{\Delta}\) by setting $d_\theta(p,q)=\theta^{s(p,q)}$.

Recall that the transfer operator $P$
corresponding to $\bar F:\bar Y\to\bar Y$ and $\bar\mu_Y$ has the form
$(P\phi)(y)=\sum_{a\in\alpha}\zeta(y_a)\phi(y_a)$.
Also, $(P^n\phi)(y)=\sum_{a\in\alpha_n}\zeta_n(y_a)\phi(y_a)$ where
$\alpha_n = \bigvee_{k=0}^{n-1} \bar{F}^{-k} \alpha $ is the partition of $\bar Y$ into $n$-cylinders and 
$\zeta_n=\zeta\,\zeta\circ \bar F\cdots \zeta\circ \bar F^{n-1}$.

\begin{prop}
  \label{prop-zetan}
  Let $a \in \alpha_n$ and $y, y' \in a$. Then
  (a) $K_1^{-1}\bar\mu_Y(a)\le \zeta_n(y)\le K_1\bar\mu_Y(a)$,
  (b) $|\zeta_n(y)-\zeta_n(y')|\le K_1\bar\mu_Y(a)d_\theta(\bar F^ny,\bar F^ny')$,
  where $K_1=e^{(1-\theta)^{-1}K}(1-\theta)^{-1}K$.
\end{prop}

\begin{proof}
  It follows from~\eqref{eq-zeta-dist} that
  \[
    |\log\zeta_n(y)-\log\zeta_n(y')|\le (1-\theta)^{-1}Kd_\theta(\bar F^ny,\bar F^ny').
  \]
  Hence
  $ \sup_a \zeta_n \le e^{(1-\theta)^{-1}K}\inf_a\zeta_n$ and
  \[ \SMALL
  \inf_a\zeta_n=\inf P^n1_a\le \int_{\bar Y}P^n1_a\,d\bar\mu_Y=\int_{\bar Y}1_a\,d\bar\mu_Y=\bar\mu_Y(a).
  \]
  Thus \(\sup_a \zeta_n \leq K_1 \bar{\mu}_Y(a)\). Similarly,
  \(\inf_a \zeta_n \geq K_1^{-1} \bar{\mu}_Y(a)\).
  Finally,
  \begin{align*}
    |\zeta_n(y)-\zeta_n(y')|
    \leq \sup_a \zeta_n \, |\log \zeta_n(y) - \log \zeta_n(y')|
    \le  K_1\bar\mu_Y(a)d_\theta(\bar F^ny,\bar F^ny').
  \end{align*}
\end{proof}

The transfer operator $L$ corresponding to $\bar f:\bar\Delta\to\bar\Delta$
and $\bar\mu_\Delta$
can be written as 
\[
  (L\phi)(p)=\sum_{\bar fq=p}g(q)\phi(q),
  \qquad \text{where} \qquad
  g(y, \ell) = 
  \begin{cases} 
    \zeta(y), & \ell = \tau(y)-1, \\
    1, & \ell < \tau(y)-1
  \end{cases}
  .
\]
Then
$(L^n\phi)(p)=\sum_{\bar f^nq=p}g_n(q)\phi(q)$
where $g_n=g\,g\circ\bar f\cdots g\circ\bar f^{n-1}$.

Define $\bar{\Delta}_0 = \{ (y, \ell) \in \bar{\Delta} \colon \ell = 0 \}$
and $\Delta_0 = \{ (y, \ell) \in \Delta \colon \ell = 0 \}$.

\begin{prop} \label{prop-gn}
Let $p,p'\in \bar{\Delta}$ with $s(p,p')\ge n\ge1$.  Then
(a) $\sum_{\bar f^nq=p}g_n(q)=1$,
(b) $|g_n(p)-g_n(p')|\le K_1^2 g_n(p)d_\theta(\bar f^np,\bar f^np')$.
\end{prop}

\begin{proof}  Part (a) is immediate since $L^n1=1$.
Let $r(p)=\#\{j\in\{1,\dots,n\}:\bar f^jp\in \bar{\Delta}_0 \}$.
Note that $r(p)=r(p')$.
If $r(p)=0$, then $g_n(p) = g_n(p')=1$
and (b) holds trivially.  Otherwise, 
we can write $p=(y,\ell)$, $p'=(y',\ell)$ with $y,y' \in \bar{Y}$
and $\ell \geq 0$.
Then $g_n(p)=\zeta_{r(p)}(y)$ and $g_n(p')=\zeta_{r(p)}(y')$.

Let $a\in\alpha_{r(p)}$ be the cylinder containing $y$ and $y'$. 
Then by Proposition~\ref{prop-zetan},
\begin{align*}
  |g_n(p)-g_n(p')|
  & = |\zeta_{r(p)}(y)-\zeta_{r(p)}(y')|
  \le K_1\bar\mu_Y(a) d_\theta(\bar F^{r(p)} y,\bar F^{r(p)} y')
  \\ & \le K_1^2 \zeta_{r(p)}(y)d_\theta(\bar f^np,\bar f^np')=K_1^2 g_n(p)d_\theta(\bar f^np,\bar f^np'),
\end{align*}
proving (b).
\end{proof}

\paragraph{Nonuniform expansion/contraction}

Recall that $\pi:\Delta\to M$ denotes
the projection $\pi(y,\ell)=T^\ell y$.
For $p=(x,\ell),q=(y,\ell)\in\Delta$, we write $q\in W^s(p)$ if $y\in W^s(x)$ and
$q\in W^u(p)$ if $y\in W^u(x)$.
Conditions (P1) translate as follows.
\begin{itemize}
\item[(P2)]
There exist constants $K_0>0$, $\rho_0\in(0,1)$ such that
for all $p,q\in \Delta$, $n\ge1$,
\begin{itemize}
\item[(i)]  If $q\in W^s(p)$, then $d(\pi f^np,\pi f^nq)\le K_0\rho_0^{\kappa_n(p)}$, and
\item[(ii)]  If $q\in W^u(p)$, then $d(\pi f^np,\pi f^nq)\le K_0\rho_0^{s(p,q)-\kappa_n(p)}$, 
\end{itemize}
\end{itemize}
where $\kappa_n(p)=\#\{j=1,\dots,n:f^jp\in \Delta_0\}$ is the number of 
returns of $p$ to $\Delta_0$ by time $n$.
It is immediate from conditions (P2) and the product structure on $Y$ that
\begin{align} \label{eq-W}
d(\pi f^n p,\pi f^n q)\le 2 K_0\rho_0^{\min\{\kappa_n(p),s(p,q)-\kappa_n(p)\}}\quad\text{for all $p,q\in\Delta$, $n\ge1$}.
\end{align}

\paragraph{Approximation of observables}
Given $C^\eta$ observables $v,w:M\to\R$, let
$\phi=v\circ \pi,\,\psi=w\circ \pi:\Delta\to\R$ be the lifted observables.
For each $n\ge1$, define $\tilde \phi_n:\Delta\to\R$,
\[
\tilde \phi_n(p)=\inf\{\phi(f^nq):s(p,q)\ge 2\kappa_n(p)\}.
\]

\begin{prop} \label{prop-tildev}
The function $\tilde \phi_n$ lies in $L^\infty(\Delta)$ and projects down to a Lipschitz 
observable $\bar \phi_n:\bar\Delta\to\R$.  Moreover, setting $K_2=1+K_1^2+2^\eta K_0^\eta$, $\rho=\rho_0^\eta$
and $\theta=\rho$,
\begin{itemize}
\item[(a)] $|\bar \phi_n|_\infty=|\tilde \phi_n|_\infty\le |v|_\infty$,
(b) $|\phi\circ f^n(p)-\tilde \phi_n(p)| \le 2^\eta K_0^\eta\|v\|_{C^\eta}\rho^{\kappa_n(p)}$ for $p\in\Delta$,
\item[(c)]  $\|L^n\bar \phi_n\|_\theta\le K_2\|v\|_{C^\eta}$, for all $n\ge1$.  
\end{itemize}
\end{prop}

\begin{proof}
This is standard, see for example~\cite[Proposition~B.5]{MT14}.  We give the details for completeness.
If $s(p,q)\ge 2\kappa_n(p)$, then $\tilde \phi_n(p)=\tilde \phi_n(q)$.
It follows that $\tilde \phi_n$ is piecewise constant on a measurable partition 
of $\Delta$, and hence is measurable, and that $\bar \phi_n$ is well-defined.
Part (a) is immediate.   

Recall that $\phi=v\circ\pi$ where $v:M\to\R$ is $C^\eta$.
Let $p\in\Delta$.
By~\eqref{eq-W} and the definition of $\tilde \phi_n$,
\begin{align*}
|\phi\circ f^n(p)-\tilde \phi_n(p)| &=|v(\pi f^np)-v(\pi f^nq)|
 \le \|v\|_{C^\eta}d(\pi f^np,\pi f^nq)^\eta
\\ & \le 2^\eta K_0^\eta\rho^{\min\{\kappa_n(p),s(p,q)-\kappa_n(p)\}},
\end{align*}
where $q$ is such that $s(p,q)\ge 2\kappa_n(p)$.
In particular, $s(p,q)-\kappa_n(p)\ge \kappa_n(p)$, so we obtain part~(b).

For part (c), first note that
$|L^n\bar \phi_n|_\infty \le |\bar \phi_n|_\infty\le |v|_\infty$.
Let $\bar p = (y, \ell) \in \bar{\Delta}$ and
$\bar p' = (y', \ell') \in \bar{\Delta}$.
If $d_\theta(\bar p,\bar p')=1$, then
\[
  \bigl|(L^n\bar{\phi})(\bar p) - (L^n\bar{\phi})(\bar p')\bigr|
  \leq 2 |v|_{\infty}
  = 2 |v|_{\infty} d_\theta(\bar p,\bar p').
\]
Otherwise, we can write
\begin{align*} 
  (L^n\bar \phi_n)(\bar p)- (L^n\bar \phi_n)(\bar p') & = I_1+I_2,
\end{align*}
where
\begin{align*}
  I_1 
  = \sum_{\bar f^n\bar q=\bar p} g_n(\bar q) \bigl(\bar \phi_n(\bar q)
    - \bar \phi_n(\bar q')\bigr),  
  \qquad 
  I_2 
  = \sum_{\bar f^n\bar q=\bar p} \bigl(g_n(\bar q)-g_n(\bar q')\bigr)
    \bar \phi_n(\bar q').
\end{align*}
As usual, preimages $\bar q,\bar q'$ are matched up so that
$s(\bar q,\bar q') = \kappa_n(\bar q) + s(\bar p,\bar p')$.

By Proposition~\ref{prop-gn}, 
\[
  \SMALL
  |I_2|
  \le K_1^2|v|_\infty
    \sum_{\bar f^n\bar q=\bar p} g_n(\bar q) 
    d_\theta(\bar f^n \bar q, \bar f^n \bar q')
  = K_1^2|v|_\infty d_\theta(\bar p,\bar p').
\]
We claim that 
$
  |\bar \phi_n(\bar q)-\bar \phi_n(\bar q')|
  \le 2^\eta K_0^\eta \|v\|_{C^\eta} \rho^{s(\bar p,\bar p')}
$.
Taking $\theta=\rho$, 
it then follows from Proposition~\ref{prop-gn}(a) that
$|I_1|\le 2^\eta K_0^\eta\|v\|_{C^\eta}d_\theta(\bar p,\bar p')$.

It remains to verify the claim.
Choose $q,q'\in\Delta$ that project onto 
$\bar q,\bar q'\in\bar\Delta$, so
\begin{align*} 
  s(q,q')=s(\bar q,\bar q')
  = \kappa_n(\bar q)+s(\bar p,\bar p').
\end{align*}
Write
$\bar \phi_n(\bar q)-\bar \phi_n(\bar q')=
\phi\circ f^n(\hat q)-\phi\circ f^n(\hat q')$,
where 
$\hat q,\hat q'\in\Delta$ satisfy
\begin{align*} \label{eq-qhat}
  s(\hat q, q)\ge 2\kappa_n(\bar q)
  \qquad \text{and} \qquad
  s(\hat q', q')\ge 2\kappa_n(\bar q).
\end{align*}
Since $\bar \phi_n(\bar q)=\bar \phi_n(\bar q')$ 
if $s(\bar q,\bar q') \ge 2\kappa_n(\bar q)$, 
we may suppose without loss that 
\begin{align*}
  s(\hat q, \hat q') = s(\bar{q},\bar q') \le 2\kappa_n(\bar q) = 2\kappa_n(\hat q).
\end{align*}
Then 
\begin{align*}
  s(\bar p,\bar p') = s(\hat q, \hat q') - \kappa_n(\hat q) \leq \kappa_n(\hat q).
\end{align*}
As in part (b),
\begin{align*}
  |\phi\circ f^n(\hat q)-\phi\circ f^n(\hat q')| 
  & \le 2^\eta K_0^\eta\|v\|_{C^\eta} 
    \rho^{\min\{\kappa_n(\hat q), \, s(\hat q,\hat q')
    - \kappa_n(\hat q)\}}
 = 2^\eta K_0^\eta\|v\|_{C^\eta} 
    \rho^{s(\bar p,\bar p')}
  .
\end{align*}
This completes the proof of the claim.
\end{proof}

\begin{cor} \label{cor-NUH}
Suppose $\{b_n\}$, $n \geq 0$ is a nonnegative non-increasing sequence,
and $|L^n\phi|_1\le b_n\|\phi\|_\theta$ for all all $n$ and all mean zero 
$d_\theta$-Lipschitz functions $\phi:\bar\Delta\to\R$. 
Then
\[
  \Big|
    \int_M v\,w\circ T^n\,d\mu-\int_M v\,d\mu\int_M w\,d\mu
  \Big|
  \le \bigl(
    2^\eta K_0^\eta|\rho^{\kappa_{[n/2]}}|_1 + 2 K_2\,b_{[n/2]}
  \bigr) \|v\|_{C^\eta}\|w\|_{C^\eta}.
\]
\end{cor}

\begin{proof}
Suppose without loss that $v$ is mean zero.
Since $\pi:\Delta\to M$ is a semiconjugacy and $\mu=\pi_*\mu_\Delta$,
it is equivalent to estimate
$\int_\Delta \phi\,\psi\circ f^n\,d\mu_\Delta$,
where $\phi, \psi \colon \Delta \to \R$,
$\phi = v \circ \pi$ and $\psi = w \circ \pi$.
Assume for simplicity that $n$ is even; the proof for $n$ odd
requires little modification.
Let $\ell\ge1$, and write
\begin{align*}
\int_{\Delta}\phi\,\psi\circ f^n\,d\mu_\Delta
=\int_{\Delta}\phi\circ f^\ell\,\psi\circ f^{\ell+n}\,d\mu_\Delta
 = I_1+I_2+I_3+I_4,
\end{align*}
where
\begin{align*}
I_1 & =\int_{\Delta}(\phi\circ f^\ell-\tilde \phi_\ell)\,\psi\circ f^{\ell+n}\,d\mu_\Delta,
\quad 
I_2 =\int_{\Delta}\tilde \phi_\ell\,(\psi\circ f^{n/2}-\tilde \psi_{n/2})\circ f^{\ell+n/2}\,d\mu_\Delta, \\
I_3 & =\int_{\Delta}\Big(\tilde \phi_\ell-\int_\Delta\tilde\phi_\ell\,d\mu_\Delta\Big)\,\tilde \psi_{n/2}\circ f^{\ell+n/2}\,d\mu_\Delta,
\quad 
I_4 =\int_{\Delta}\tilde \phi_\ell\,d\mu_\Delta\int_\Delta\tilde \psi_{n/2}\,d\mu_\Delta.
\end{align*}

By Proposition~\ref{prop-tildev}(b), 
$|I_1|\le|\phi\circ f^\ell-\tilde \phi_\ell|_1 |\psi|_\infty
\le 2^\eta K_0^\eta |\rho^{\kappa_\ell}|_1\|v\|_{C^\eta}|w|_\infty$.
By Proposition~\ref{prop-tildev}(a,b), 
$|I_2|\le|\tilde \phi_\ell|_\infty |\psi\circ f^{n/2}-\tilde \psi_{n/2}|_1
\le 2^\eta K_0^\eta |v|_\infty \|w\|_{C^\eta}|\rho^{\kappa_{n/2}}|_1$.
By Proposition~\ref{prop-tildev}(c),
\begin{align*}
|I_3|
& =\Bigl|\int_{\bar\Delta}L^{n/2}\Big(L^{\ell}\bar \phi_\ell-\int_{\bar\Delta}\bar\phi_\ell\,d\bar\mu_\Delta\Big)\,\bar \psi_{n/2}\,d\bar\mu_\Delta\Bigr| \\
 & \le |L^{n/2}( L^\ell\bar \phi_\ell-{\SMALL\int}_{\bar\Delta}\bar\phi_\ell\,d\bar\mu_\Delta)|_1 |\bar \psi_{n/2}|_\infty
\le 2 b_{n/2} \|L^\ell \bar \phi_\ell\|_\theta |w|_\infty
\le 2 K_2b_{n/2}\|v\|_{C^\eta} |w|_\infty.
\end{align*}
Finally,
$|I_4|\le |\int_{\bar\Delta}\tilde \phi_\ell\,d\bar\mu_{\Delta}| |w|_\infty=
|\int_{\bar\Delta}(\tilde \phi_\ell-\phi\circ f^\ell)\,d\bar\mu_{\Delta}||w|_\infty
\le 2^\eta K_0^\eta \|v\|_{C^\eta}|w|_\infty |\rho^{\kappa_\ell}|_1$ by
another application of Proposition~\ref{prop-tildev}(b).

Altogether, 
\[
  \Bigl|
    \int_{\Delta}\phi\,\psi\circ f^n\,d\mu_\Delta
  \Bigr|
  \le \bigl(
    2^\eta K_0^\eta|\rho^{\kappa_{n/2}}|_1
    + 2 K_2\,b_{n/2}
    + 2^{\eta + 1} K_0^\eta|\rho^{\kappa_\ell}|_1
  \bigr) \|v\|_{C^\eta}\|w\|_{C^\eta}
    .
\]
Letting $\ell\to\infty$ yields the result.
\end{proof}

By Corollary~\ref{cor-NUH} and Theorem~\ref{thm-NUE}, it remains to estimate
$|\rho^{\kappa_n}|_1$. A first step towards this is:

\begin{lemma} \label{lem-tilde}
  $
    \int_{\bar\Delta}\rho^{\kappa_n}\,d\bar\mu_\Delta 
    \le 2 \bar{\tau}^{-1} \sum_{j>n/3}\bar\mu_Y(\tau \geq j)
      + n\sum_{k=0}^\infty \rho^{k+1}\bar\mu_Y(\tau_{k}\ge n/3)
  $,
  where $\tau_k = \sum_{j=0}^{k-1} \tau \circ \bar{F}^k$.
\end{lemma}

\begin{proof}
First write $\int_{\bar\Delta}\rho^{\kappa_n}\,d\bar\mu_\Delta=\sum_{k=0}^\infty \rho^k \bar\mu_\Delta(\kappa_n=k)$.
Note that $\kappa_n(p)=0$ if and only if $f^j(p)\not\in \bar{\Delta}_0$ for all $j=1,\dots,n$,
so $\bar\mu_\Delta(\kappa_n=0)=\bar\tau^{-1}\sum_{j\geq n}\bar\mu_Y(\tau>j)$.

When $\kappa_n(p)\ge1$, we can define $r(p)=\min\{j\in\{1,\dots,n\}: f^jp\in \bar{\Delta}_0\}$
and $s(p)=\max\{j\in\{1,\dots,n\}: f^jp\in \bar{\Delta}_0\}$.
Hence for $k\ge1$,
\[
  \{\kappa_n(p)=k\}=\bigcup_{1\le r\le s\le n}
  \{\kappa_n(p)=k,\, r(p)=r,\,s(p)=s\}.
\]
It is easy to check that 
$\bar \mu_\Delta\{r(p)=j\} = \bar\tau^{-1}\bar\mu_Y(\tau \geq j)$, so 
\[
  \bar\mu_\Delta(\kappa_n(p)=k)
  \le \bar\tau^{-1}\sum_{j>n/3}\bar\mu_Y(\tau \geq j)+ b_{n,k},
\]
where
\begin{align*}
b_{n,k} &=
\sum_{0\le r\le n/3}\sum_{2n/3\le s\le n}\bar\mu_\Delta(\kappa_n(p)=k,\,r(p)=r,\,s(p)=s)
\\ & =
\sum_{0\le r\le n/3}\sum_{2n/3\le s\le n}\bar\mu_\Delta(\kappa_{s-r}(f^rp)=k - 1,\,r(p)=r,\,s(p)=s)
\\ & \le 
\sum_{0\le r\le n/3}\sum_{2n/3\le s\le n}\bar\mu_\Delta(\kappa_{s-r}(f^rp)=k - 1,\,f^rp\in\bar{\Delta}_0,\,f^sp\in\bar{\Delta}_0)
\\ & =
\sum_{0\le r\le n/3}\sum_{2n/3\le s\le n}\bar\mu_\Delta(\kappa_{s-r}(p)=k - 1,\,p\in\bar{\Delta}_0,\,f^{s-r}p\in\bar{\Delta}_0)
\\ & \le 
n\sum_{j\ge n/3} \bar\mu_\Delta(\kappa_j(p)=k - 1,\,p\in\bar{\Delta}_0,\,f^jp\in\bar{\Delta}_0)
\\ & =
n\bar\tau^{-1}\sum_{j\ge n/3} \bar\mu_Y(y\in\bar Y,\,f^jy\in\bar Y,\,\tau_{k-1}(y)=j)
 \le 
n\bar\tau^{-1}\bar\mu_Y(y\in\bar Y:\tau_{k-1}(y)\ge n/3).
\end{align*}
This completes the proof.
\end{proof}

\begin{pfof}{Theorem~\ref{thm-NUH}}
We restrict from now on to the cases of polynomial tails and stretched exponential tails.  The sum $\sum_{j\ge n}\bar\mu_Y(\tau>j)$ is estimated in the same way
as $\bP(w_1\ge n)$ in the proofs of Propositions~\ref{prop-poly} and~\ref{prop-aexp},
so it remains to show that
$n\sum_{k=0}^\infty \rho^k\bar\mu_Y(\tau_k\ge n)$ satisfies the required estimate.

In the case of polynomial tails, 
$\bar\mu_Y(\tau_k\ge n)\le k\bar\mu_Y(\tau\ge n/k)\le C_\tau  k^{\beta+1}n^{-\beta}$,
and so
$n\sum_{k=1}^\infty \rho^k\bar\mu_Y(\tau_k\ge n)\le C_2n^{-(\beta-1)}$ where
$C_2=C_\tau \sum_{k=1}^\infty \rho^k k^{\beta+1}$.

It remains to treat the stretched exponential case.
Writing $X_j=\tau\circ F_j$, 
\begin{align*}
 \bar\mu_Y(X_0=j_0,\dots  ,X_k=j_k)
 & = \int_{\bar Y} 1_{\{\tau\circ F^k=j_k\}}1_{\{X_0=j_0,\dots,X_{k-1}=j_{k-1}\}}\,d\bar\mu_Y
 \\ & =\int_{\bar Y} 1_{\{\tau=j_k\}}P^k1_{\{X_0=j_0,\dots,X_{k-1}=j_{k-1}\}}\,d\bar\mu_Y
 \\ & \le \bar\mu_Y(\tau=j_k)|P^k1_{\{X_0=j_0,\dots,X_{k-1}=j_{k-1}\}}|_\infty
\end{align*}
By~Proposition~\ref{prop-zetan}(a),
\begin{align*}
  (P^k1_{\{X_0=j_0,\dots  ,X_{k-1}=j_{k-1}\}} & )(y)=
  \sum_{a\in\alpha_k}\zeta_k(y_a)1_{\{X_0=j_0,\dots,X_{k-1}=j_{k-1}\}}
  \\ & \le K_1\sum_{a\in\alpha_k}\bar\mu_Y(a)1_{\{\tau(a)=j_0,\dots,\tau(F^{k-1}a)=j_{k-1}\}}
  \\ & = K_1\bar{\mu}_Y(\tau=j_0,\dots,\tau\circ F^{k-1}=j_{k-1}).
\end{align*}
Hence
\[
  \bar{\mu}_Y(X_0=j_0,\dots,X_k=j_k)
  \le K_1 \bar{\mu}_Y(\tau=j_k) \bar{\mu}_Y(X_0=j_0,\dots,X_{k-1}=j_{k-1}),
\]
and so
\[
  \bar{\mu}_Y(X_k\ge n\,|\, X_0=j_0,\dots,X_{k-1}=j_{k-1})
  \le K_1 \bar{\mu}_Y(\tau\ge n)
  \le K_1 C_\tau e^{-An^\gamma}.
\]
By Proposition~\ref{prop-aexp}, there exists $B\in(0,A)$ and $C_B\in(0,\rho)$
depending continuously on $C_\tau$, $\gamma$ and $A$ such that
\[
  \mu_Y(\tau_k\ge n)
  = \mu_Y(X_0+\dots + X_{k-1}\ge n)\le C_B^k e^{-B n^\gamma},
\]
Hence $\sum_{k=1}^\infty \rho^k\mu_Y(\tau_k\ge n)
\le \{\sum_{k=1}^\infty (\rho C_B)^k\} e^{-B n^\gamma}$
as required.
\end{pfof}

\paragraph{Acknowledgements}
This research was supported in part by a
European Advanced Grant {\em StochExtHomog} (ERC AdG 320977).


\begin{thebibliography}{1}

\bibitem{Aaronson}
J.~Aaronson. \emph{{An Introduction to Infinite Ergodic Theory}}. Math. Surveys
  and Monographs \textbf{50}, Amer. Math. Soc., 1997.

\bibitem{ADU93}
J.~Aaronson, M.~Denker and M.~Urba{\'n}ski. Ergodic theory for {M}arkov fibred
  systems and parabolic rational maps. \emph{Trans. Amer. Math. Soc.}
  \textbf{337} (1993) 495--548.

\bibitem{Alves04}
J.~F. Alves. Strong statistical stability of non-uniformly expanding maps.
  \emph{Nonlinearity} \textbf{17} (2004) 1193--1215.

\bibitem{AlvesAzevedo16}
J.~F. Alves and D.~Azevedo. Statistical properties of diffeomorphisms with weak
  invariant manifolds. \emph{Discrete Contin. Dyn. Syst.} \textbf{36} (2016)
  1--41.

\bibitem{AlvesLuzzattoPinheiro05}
J.~F. Alves, S.~Luzzatto and V.~Pinheiro. Markov structures and decay of
  correlations for non-uniformly expanding dynamical systems. \emph{Ann. Inst.
  H. Poincar\'e Anal. Non Lin\'eaire} \textbf{22} (2005) 817--839.

\bibitem{AlvesPinheiro08}
J.~F Alves and V.~Pinheiro Slow rates of mixing for dynamical systems with
  hyperbolic structures \emph{J. Stat. Phys.} \textbf{131} (2008) 505--534.

\bibitem{AlvesViana02}
J.~F. Alves and M.~Viana. Statistical stability for robust classes of maps with
  non-uniform expansion. \emph{Ergodic Theory Dynam. Systems} \textbf{22}
  (2002) 1--32.

\bibitem{BaladiBenedicksSchnellmann15}
V.~Baladi, M.~Benedicks and D.~Schnellmann. Whitney-{H}\"older continuity of
  the {SRB} measure for transversal families of smooth unimodal maps.
  \emph{Invent. Math.} \textbf{201} (2015) 773--844.

\bibitem{Bowen75}
R.~Bowen. \emph{{Equilibrium States and the Ergodic Theory of Anosov
  Diffeomorphisms}}. Lecture Notes in Math. \textbf{470}, Springer, Berlin,
  1975.

\bibitem{BL} X. Bressaud and C. Liverani. Anosov diffeomorphisms and coupling. \emph{Ergod. Th. Dynam. Syst.} \textbf{22} (2002) 129--152.

\bibitem{ChazottesGouezel12} J.-R. Chazottes and S. Gou\"ezel.
Optimal concentration inequalities for dynamical systems.
{\em Commun. Math. Phys.} \textbf{316} (2012) 843--889.

\bibitem{CD} N. Chernov and D. Dolgopyat. Brownian Brownian Motion -- I. \emph{Memoirs Amer. Math. Soc.} \textbf{198}, No.\ 927, 2009.

\bibitem{DemersZhang13}
M.~F. Demers and H.-K. Zhang. A functional analytic approach to perturbations
  of the {L}orentz gas. \emph{Commun. Math. Phys.} \textbf{324} (2013) 767--830.


\bibitem{Eslamiapp}
P.~Eslami. {Stretched exponential mixing for $C^{1+\alpha}$ skew products with
  discontinuities}. \emph{Ergodic Theory Dynam. Systems} \textbf{37} (2017) 146--175.

\bibitem{FreitasTodd09}
J.~M. Freitas and M.~Todd. The statistical stability of equilibrium states for
  interval maps. \emph{Nonlinearity} \textbf{22} (2009) 259--281.

\bibitem{Gouezel06}
S.~Gou{\"e}zel. Decay of correlations for nonuniformly expanding systems.
  \emph{Bull. Soc. Math. France} \textbf{134} (2006) 1--31.

\bibitem{GouezelPhD}
S.~Gou{\"e}zel. {Vitesse de d\'ecorr\'elation et th\'eor\`emes limites pour les
  applications non uniform\'ement dilatantes}. {\mbox{Ph.\ D.} Thesis}, Ecole
  Normale Sup\'erieure, 2004.


\bibitem{KellerLiverani99}
G.~Keller and C.~Liverani. {Stability of the spectrum for transfer operators}.
  \emph{Annali della Scuola Normale Superiore di Pisa, Classe di Scienze}
  \textbf{XXVIII} (1999) 141--152.

\bibitem{KKMsub1}
A.~Korepanov, Z.~Kosloff and I.~Melbourne. Averaging and rates of averaging
  for uniform families of deterministic fast-slow skew product systems.
  \emph{Studia Math.} \textbf{238} (2017) 59--89.

\bibitem{KKMsub2}
 A. Korepanov, Z. Kosloff and I. Melbourne.
Martingale-coboundary decomposition for families of dynamical systems.
\emph{Ann. Inst. H. Poincar\'e Anal. Non Lin\'eaire}
\textbf{35} (2018) 859--885.

\bibitem{Liverani01}
C.~Liverani. Rigorous numerical investigation of the statistical properties of
  piecewise expanding maps. {A} feasibility study. \emph{Nonlinearity}
  \textbf{14} (2001) 463--490.

\bibitem{Maume01}
V.~Maume-Deschamps. {Projective metrics and mixing properties on towers}.
  \emph{Trans. Amer. Math. Soc.} \textbf{353} (2001) 3371--3389.

\bibitem{MN05}
I.~Melbourne and M.~Nicol. Almost sure invariance principle for nonuniformly
  hyperbolic systems. \emph{Commun. Math. Phys.} \textbf{260} (2005) 131--146.

\bibitem{MT14}
I.~Melbourne and D.~Terhesiu. {Decay of correlations for nonuniformly expanding
  systems with general return times}. \emph{Ergodic Theory Dynam. Systems}
  \textbf{34} (2014) 893--918.

\bibitem{ParryPollicott90}
W.~Parry and M.~Pollicott. \emph{{Zeta Functions and the Periodic Orbit
  Structure of Hyperbolic Dynamics}}. Ast{\'e}rique \textbf{187-188},
  Soci{\'e}t{\'e} Math{\'e}matique de France, Montrouge, 1990.

\bibitem{Ruelle78}
D.~Ruelle. \emph{{Thermodynamic Formalism}}. Encyclopedia of Math. and its
  Applications \textbf{5}, Addison Wesley, Massachusetts, 1978.

\bibitem{Selmer77}
E.~Selmer. {On the linear diophantine problem of Frobenius}. 
  \emph{J. Reine Angew. Math.}
  \textbf{293/294} (1977) 1--17.

\bibitem{Sinai72}
Y.~G. Sina{\u\i}. {Gibbs measures in ergodic theory}. \emph{Russ. Math. Surv.}
  \textbf{27} (1972) 21--70.

\bibitem{Sulku}
H.~Sulku. Explicit correlation bounds for expanding circle maps using the
  coupling method. Preprint, 2013.

\bibitem{Viana97}
M.~Viana. Multidimensional nonhyperbolic attractors. \emph{Inst. Hautes
  \'Etudes Sci. Publ. Math.} (1997) 63--96.

\bibitem{Young98}
L.-S. Young. Statistical properties of dynamical systems with some
  hyperbolicity. \emph{Ann. of Math.} \textbf{147} (1998) 585--650.

\bibitem{Young99}
L.-S. Young. Recurrence times and rates of mixing. \emph{Israel J. Math.}
  \textbf{110} (1999) 153--188.

\bibitem{Zweimuller04}
R.~Zweim{\"u}ller. Kuzmin, coupling, cones, and exponential mixing. \emph{Forum
  Math.} \textbf{16} (2004) 447--457.

\end{thebibliography}
\end{document}